\documentclass[11pt,a4paper]{amsart}

\usepackage{mymacros_english_paper}
\usepackage{mymacros_common}
\begin{document}
\title[Prismatic Kunz's Theorem]{Prismatic Kunz's Theorem}

\author[R. Ishizuka]{Ryo Ishizuka}
\address{Department of Mathematics, Institute of Science Tokyo, 2-12-1 Ookayama, Meguro, Tokyo 152-8551}
\email{ishizuka.r.ac@m.titech.ac.jp}

\author[K. Nakazato]{Kei Nakazato}
\address{Proxima Technology Inc., 8-5-7 Minamisenju, Arakawa, Tokyo 116-0003}
\email{keinakazato31@gmail.com}



\begin{abstract}
    In this paper, we prove ``prismatic Kunz's theorem'' which states that a complete Noetherian local ring \(R\) of residue characteristic \(p\) is a regular local ring if and only if the Frobenius lift on a prismatic complex of (a derived enhancement of) \(R\) over a specific prism \((A, I)\) is faithfully flat.
    This generalizes classical Kunz's theorem from the perspective of extending the ``Frobenius map'' to mixed characteristic rings.
    Our approach involves studying the deformation problem of the ``regularity'' of prisms and demonstrating the faithful flatness of the structure map of the prismatic complex.
\end{abstract}

\maketitle

\setcounter{tocdepth}{1}
\tableofcontents
\section{Introduction} \label{Introduction}

Let \(p\) be a prime number.
In this paper, we assume that any ring is a (commutative) \(\setZ_{(p)}\)-algebra.

What is the ``\emph{Frobenius map}'' on local rings of residue characteristic \(p\)?
In positive characteristic \(p\), the ``\emph{Frobenius map}'' is exactly the Frobenius map \(F \colon x \mapsto x^p\).
The existence of this map is essential in the theory of positive characteristic rings, for example, \(F\)-singularities and tight closure theory.
In particular, the following Kunz's theorem is a monumental result that marks the beginning of the idea.

\begin{theorem}[{Kunz's theorem; \cite{kunz1969Characterizations}}] \label{KunzFrob}
    Let \(R\) be a Noetherian ring of characteristic \(p\).
    Then \(R\) is regular if and only if the Frobenius map \(\map{F}{R}{F_*R}\) is faithfully flat if and only if the canonical map \(R \to R_{\perf}\) is faithfully flat, where \(R_{\perf}\) is the perfect closure \(R_{\perf} \defeq \colim_F R\) of \(R\).
\end{theorem}

By considering the case of mixed characteristic \((0, p)\), the theory of positive characteristic rings has been extended to the ``\emph{\(p\)-adic commutative algebra theory}'', that is the theory of local rings of residue characteristic \(p\).
However, in mixed characteristic, we cannot define the ``\emph{Frobenius map}'' in the same way, and this lack of the mysterious map \(F\) is a big problem.

Instead of searching for the ``\emph{Frobenius map}'', the notion of \emph{perfectoid rings} - a generalization of perfect rings to mixed characteristic introduced by Scholze \cite{scholze2012Perfectoida} - has been applied in commutative algebra after Andr\'e \cite{andre2018Lemme,andre2018Conjecture} and Bhatt \cite{bhatt2014Almosta,bhatt2018Direct}.
Recently, as perfectoid theory progressed, Bhatt--Scholze \cite{bhatt2022Prismsa} developed the theory of \emph{prisms} and \emph{prismatic complexes}, enabling the exploration of the generalization of the Frobenius map.

We first briefly recall these notions.
A (\(p\)-torsion-free) \emph{\(\delta\)-ring \(A\)} is a (\(p\)-torsion-free) \(\setZ_{(p)}\)-algebra \(A\) equipped with an endomorphism \(\map{\varphi_A}{A}{A}\) called a Frobenius lift such that \(\varphi_A\) induces the Frobenius map \(F\) on \(A/pA\).
A pair \((A, I)\) of a \(\delta\)-ring \(A\) and its ideal \(I\) is a \emph{prism} if \(A\) is derived \((p, I)\)-complete, \(I\) defines a locally free \(A\)-module of rank \(1\), and \(p \in I + \varphi_A(I)A\) (see \Cref{DefDeltaRing} and \Cref{DefPrism}).
If \(R\) is an \(A/I\)-algebra, we can define the \emph{prismatic complex \(\prism_{R/A}\) of \(R\) relative to \(A\)} equipped with the \(\varphi_A\)-semilinear endomorphism \(\varphi\) on \(\prism_{R/A}\) (see \Cref{PrismaticComplexHodgeTate}).

Under this notation, the philosophy of the generalization of the Frobenius map can be captured by the following slogan:

\begin{slogan}[{cf. \cite[Remark 5.6]{bhatt2022Algebraic} and \cite[\S 1.4]{bhatt2021CohenMacaulayness}}] \label{Slogan}
    Let \(R\) be a ``good'' Noetherian local ring of residue characteristic \(p\).
    The \emph{``Frobenius map''} on \(R\) is the \(\varphi_A\)-semilinear endomorphism \(\varphi\) on a \emph{prismatic complex} \(\prism_{R/A}\) for some prism \((A, I)\) such that \(R\) is an \(A/I\)-algebra.
\end{slogan}

Under this philosophy, Bhatt showed the Cohen-Macaulayness of absolute integral closures \cite{bhatt2021CohenMacaulayness}.
Our purpose in this paper is to give one of the formations of ``\emph{prismatic Kunz's theorem}'' by using this ``\emph{Frobenius map}'' (\Cref{maintheoremPrismaticKunz}).
To prove this, we study the deformation problem of the ``regularity'' of prisms.

\begin{theorem}[{\Cref{EquivRegularPrism}}] \label{maintheoremDeformation}
    Let \((A, I)\) be a bounded prism (here we do not yet assume that \(A\) is Noetherian).
    We say that \((A, I)\) is a \emph{regular local prism} if it satisfies one of the following equivalent conditions:
    \begin{enumerate}
        \item[(1)] \(A\) is a regular local ring.
        \item[(1')] \(A\) is an unramified regular local ring, namely, \((A, \mfrakm_A)\) is a regular local ring and \(p \in \mfrakm_A \setminus \mfrakm_A^2\).
        \item[(2)] \(A/pA\) is a regular local ring.
        \item[(3)] \(A/I\) is a regular local ring.
    \end{enumerate}
\end{theorem}

This is a useful criterion to check whether a prism is regular or not.
By only using this, we can show a one-to-one correspondence (\Cref{CorrespCompleteRegularLocalRing}) between the set of \(d\)-dimensional complete regular local rings with residue field \(k\) of characteristic \(p\) and the set of prismatic structures of the formal power series ring \(C(k)[|T_1, \dots, T_n|]\) whose coefficient ring \(C(k)\) is the Cohen ring of \(k\) defined in \Cref{CohenRing}.

We frequently encounter not only a prism \((A, I)\) itself but also an \(A/I\)-algebra \(R\).
Two common cases are as follows (for additional context, refer to \Cref{RemarkBoundedAssumption} and \Cref{RemarkPrismAssumption}).

One is the case of when \(R\) is a semiperfectoid ring (i.e., \(R\) is derived \(p\)-complete and is a quotient of a perfectoid ring).
In this case, \(R\) admits a surjective map \(A \twoheadrightarrow R\) with a perfect prism \((A, I)\) such that \(I\) is contained in \(\ker(A \twoheadrightarrow R)\).
This situation often arises in arithmetic geometry and is one of the comprehensible cases within the theory of prismatic complexes and perfectoidization (refer to \cite[\S 7]{bhatt2022Prismsa} and \cite{ishizuka2023Calculation} for further details).

Another case occurs in commutative ring theory: Any complete Noetherian local ring \(R\) admits a surjective \(A \twoheadrightarrow R\) with a complete regular local prism \((A, I)\) (i.e., \((A, I)\) is a prism, and \(A\) is a complete regular local ring: Refer to the equivalence in \Cref{maintheoremDeformation} above) such that \(I\) is contained in \(\ker(A \twoheadrightarrow R)\).
This result stems from Cohen's structure theorem (see \Cref{DefofRegularPrism} and \Cref{Envelope} for more details).

The second main theorem (\Cref{maintheoremFFlat}) covers these cases.

\subsection{Faithful flatness of prismatic complexes}

One of our main theorems is the following general theory of prismatic complexes.
This is a generalization of \cite[Example VII.4.4]{bhatt2018Geometric} which requires the assumption that \(R\) becomes a regular semiperfectoid ring with a perfect prism \((A, I)\).
To formulate this theorem, we use the notion of \emph{animated rings} (see \Cref{AnimatedAppendix} for the basic knowledge of animated rings).
In particular, the following ``derived enhancement'' of a quotient ring is a key tool in our proof.

\begin{construction}[{\Cref{ConstructionAnimatedQuotient}}]
    Let \((A, I)\) be a bounded prism and let \(R \defeq A/J\), where \(J = (I, f_1, \dots, f_r)\) with a (not necessarily regular) sequence of elements \(f_1, \dots, f_r\) in \(A\).
    We can define an animated \(A/I\)-algebra by taking a derived quotient
    \begin{equation*}
        R^{an}(\underline{f}) \defeq R^{an}(f_1, \dots, f_r) \defeq (A/I)/^L(f_1, \dots, f_r) = A/I \otimes^L_{\setZ[X_1, \dots, X_r]} \setZ,
    \end{equation*}
    where \(A/I \leftarrow \setZ[X_1, \dots, X_r] \rightarrow \setZ\) is defined by \(f_i \mapsfrom X_i \mapsto 0\).

    As a complex, this is equal to the Koszul complex of \(f_1, \dots, f_r\) over \(A/I\) and especially this depends on the choice of the elements \(f_1, \dots, f_r\).
    However, our prismatic Kunz's theorem (\Cref{maintheoremPrismaticKunz}) does not depend on a specific choice of \(f_1, \dots, f_r\) and thus one can apply the theorem with any choice one finds convenient (\Cref{WarningChoice}).
\end{construction}

In order to prove prismatic Kunz's theorem, we need to show the faithful flatness of the structure map of the prismatic complex.

\begin{theorem}[{\Cref{AnimatedLCIPrism} and \Cref{NoetherianCase}}] \label{maintheoremFFlat}
    Let \((A, I)\) be a bounded prism and let \(R \defeq A/J\), where \(J = (I, f_1, \dots, f_r)\) with a (not necessarily regular) sequence of elements \(f_1, \dots, f_r\) in \(A\).
    Set an animated ring \(R^{an}(\underline{f})\) as above.
    Then the canonical maps of rings
    \begin{equation*}
        R/p^nR \to \pi_0(\overline{\prism}_{R^{an}(\underline{f})/A})/(p^n)
    \end{equation*}
    are faithfully flat for all \(n \geq 1\), where \(\pi_0(\overline{\prism}_{R^{an}(\underline{f})/A})\) is the connected component of the Hodge--Tate complex \(\overline{\prism}_{R^{an}(\underline{f})/A} = \prism_{R^{an}(\underline{f})/A} \otimes^L_A A/I\), that is, the \(0\)-th cohomology of \(\overline{\prism}_{R^{an}(\underline{f})/A}\) when we regard it as a commutative algebra object in the derived category \(D(R^{an}(\underline{f}))\) of \(R^{an}(\underline{f})\)-modules.
    Moreover, if \(R\) is Noetherian, the map \(R \to \pi_0(\overline{\prism}_{R^{an}(\underline{f})/A})\) itself is faithfully flat.
\end{theorem}

This shows that some properties of \(\pi_0(\overline{\prism}_{R^{an}(\underline{f})/A})\) can descend to \(R\).
So the deep theory of prismatic complexes can be applied to the study of \(R\) involving most semiperfectoid rings and complete Noetherian local rings.

\subsection{Prismatic Kunz's theorem}

Applying \Cref{maintheoremFFlat} for commutative algebraic situation, we show the following formation of ``prismatic Kunz's theorem''.

\begin{theorem}[{Prismatic Kunz's theorem (\Cref{RegularGivesFFlat}, \Cref{FFlatGivesRegular}, and \Cref{AnimatedCase})}] \label{maintheoremPrismaticKunz}
    Let \((R, \mfrakm, k)\) be a complete Noetherian local ring of residue characteristic \(p\).
    For a fixed minimal generator \(x_1, \dots, x_n\) of \(R\), Cohen's structure theorem makes a surjective map
    \begin{equation*}
        A \defeq C(k)[|T_1, \dots, T_n|] \twoheadrightarrow R,
    \end{equation*}
    where \(C(k)\) is the Cohen ring of \(k\).
    Then there exists an ideal \(I\) of \(A\) contained in \(\ker(A \twoheadrightarrow R)\) such that \((A, I)\) becomes a prism (\Cref{Envelope}).
    Fix any sequence of elements \(f_1, \dots, f_r\) in \(A\) such that \(\ker(A \twoheadrightarrow R) = (I, f_1, \dots, f_r)\).
    Set an animated ring \(R^{an}(\underline{f})\).
    Under this construction, this theorem states that the following conditions are equivalent:
    \begin{enumerate}
        \item \(R\) is a regular local ring.
        \item The Frobenius lift \(\map{\varphi}{\prism_{R^{an}(\underline{f})/A}}{\varphi_{A,*}\prism_{R^{an}(\underline{f})/A}}\) of the (animated) \(\delta\)-\(A\)-algebra \(\prism_{R^{an}(\underline{f})/A}\) is faithfully flat.
    \end{enumerate}
\end{theorem}

This theorem is proved by our second main theorem (\Cref{maintheoremFFlat}) and \(p\)-adic Kunz's theorem developed by Bhatt--Iyengar--Ma \cite{bhatt2019Regular} (see \Cref{KunzRegular}).
Using this insight, we propose the Frobenius lift on \(\prism_{R^{an}(\underline{f})/A}\) as the ``\emph{Frobenius map}'' for complete Noetherian local rings with residue characteristic \(p\) as stated in \Cref{Slogan}.

Although this theorem is itself an application of the second theorem (\Cref{maintheoremFFlat}), this gives us two applications.
The first application is a characterization of the regularity of prisms by using the faithful flatness of the Frobenius lift (\Cref{PrismaticKunzCor}).
The second application is a new proof of the stability of regularity of complete Noetherian local rings under the localization of prime ideals that contains \(p\) (\Cref{KunzAnotherProof}).

Furthermore, as an alternative result, it is possible to prove the first application (\Cref{PrismaticKunzCor}) without using \Cref{maintheoremFFlat}, solely relying on classical Kunz's theorem.
In this case, this \Cref{KunzAnotherProof} concerning the stability of regularity can be proved without Serre's regularity criterion and even without \(p\)-adic Kunz's theorem.

\subsection*{Outline}
We begin in \Cref{SectionTransversalPrisms} by recalling the notion of (transversal) prisms and giving some basic properties.
In \Cref{SectionRegularPrisms}, we show the deformation property of the regularity of prisms (\Cref{maintheoremDeformation}).
\Cref{maintheoremFFlat} is proven in \Cref{SectionPrismaticComplexes} in which we recall the notion of prismatic complexes.
Our main purpose of this paper, prismatic Kunz's theorem (\Cref{maintheoremPrismaticKunz}), is proven in \Cref{SectionPrismaticKunz} by using \Cref{maintheoremFFlat} and \(p\)-adic Kunz's theorem recalled in \Cref{KunzRegular}.
We freely use the language of higher algebra such as \(\infty\)-categories and animated rings after \Cref{SectionPrismaticComplexes} which we briefly summarize in \Cref{AnimatedAppendix}.

\subsection*{Acknowledgments}
The first-named author would like to express his sincere gratitude to Dimitri Dine, Tetsushi Ito, and Teruhisa Koshikawa for their valuable conversations.
The authors are deeply thankful to Kazuki Hayashi, Shinnosuke Ishiro, and Kazuma Shimomoto for their continuous support.
Special thanks to Bhargav Bhatt for reviewing the early drafts, especially a counterexample \Cref{CounterExample}, and pointing out a mistake in the proof of \Cref{AnimatedLCIPrism}.
The first-named author was partially supported by JSPS KAKENHI Grant Number 24KJ1085.
The second-named author was partially supported by JSPS Grant-in-Aid for Early-Career Scientists 23K12952.

\section{Transversal Prisms} \label{SectionTransversalPrisms}

In this section, we introduce the notion of transversal prisms, which is defined in \cite{anschutz2020Pcompleted}.
First, we recall the definition of \(\delta\)-rings and prisms.

\begin{definition}[{\cite[Definition 2.1]{bhatt2022Prismsa}}] \label{DefDeltaRing}
    Let \(A\) be a ring.
    A \emph{\(\delta\)-structure on \(A\)} is a map of sets \(\map{\delta}{A}{A}\) such that \(\delta(0) = \delta(1) = 0\) and
    \begin{equation*}
        \delta(a + b) = \delta(a) + \delta(b) + \frac{a^p + b^p - (a + b)^p}{p}; \delta(ab) = a^p \delta(b) + b^p \delta(a) + p\delta(a)\delta(b)
    \end{equation*}
    for all \(a, b \in A\).
    A \emph{\(\delta\)-ring} is a pair \((A, \delta)\) of a ring \(A\) and a \(\delta\)-structure on \(A\).
    We often omit the \(\delta\)-structure \(\delta\) and simply say that \(A\) is a \(\delta\)-ring.
    An element \(a \in A\) is called a \emph{distinguished element} (resp., \emph{of rank-\(1\)}) if \(\delta(a)\) is invertible in \(A\) (resp., \(\delta(a) = 0\)).

    On a \(\delta\)-ring \(A\), a map of sets \(\map{\varphi}{A}{A}\) is defined as
    \begin{equation*}
        \varphi(a) \defeq a^p + p \delta(a)
    \end{equation*}
    for all \(a \in A\).
    By the definition of \(\delta\), \(\varphi\) gives a ring endomorphism and we call it the \emph{Frobenius lift} on the \(\delta\)-ring \(A\).
    The induced map on \(A/pA\) becomes the usual Frobenius map \(\map{F}{A/pA}{A/pA}\).

    We often use the symbol \(\varphi_*(-)\) and \(F_*(-)\) as the restriction of scalars along \(\varphi\) and \(F\), respectively.
\end{definition}

We recall the definition of the derived completeness.
The derived \(\mfraka\)-completeness of modules over animated rings is also introduced in \cite[\S 7.3]{lurie2018Spectral} and we recall it in \Cref{DefDerivedCompletion}.

\begin{definition}[{cf. \citeSta{091N} and \cite[\S 1.2]{bhatt2022Prismsa}}] \label{DefDerivedComplete}
    Let \(A\) be a ring and let \(\mfraka = (f_1, \dots, f_r)\) be a finitely generated ideal of \(A\).
    A complex \(M\) of \(A\)-modules is \emph{derived \(\mfraka\)-complete} if the canonical map
    \begin{equation*}
        M \to \widehat{M} \defeq R\lim_n (M \otimes^L_A \Kos(A; f_1^n, \dots, f_r^n))
    \end{equation*}
    is an isomorphism in \(D(A)\), where \(\Kos(A; f_1^n, \dots, f_r^n)\) is the Koszul complex of \(A\) with respect to \(f_1^n, \dots, f_r^n\).
    This derived limit \(\widehat{M}\) is called the \emph{derived \(\mfraka\)-completion} of \(M\).

    An \(A\)-module \(M\) is \emph{derived \(\mfraka\)-complete} if the complex \(M[0]\) is derived \(\mfraka\)-complete as a complex of \(A\)-modules.
    We say that \(A\) is \emph{derived \(\mfraka\)-complete} if \(A\) is derived \(\mfraka\)-complete as an \(A\)-module.
    A complex \(M\) is derived \(\mfraka\)-complete if and only if each \(H^i(M)\) is derived \(\mfraka\)-complete for all \(i \in \setZ\).
\end{definition}

Some properties of derived \(\mfraka\)-completeness are summarized in, for example, \citeSta{019N}, \cite[\S 1.2]{bhatt2022Prismsa}, \cite[Lecture III]{bhatt2018Geometric}, and \cite[\S 6]{kedlayaNotes}.
For convenience, we recall the following lemma.

\begin{lemma}[{\cite[Corollary 6.3.2]{kedlayaNotes}}] \label{LemmaDerivedComplete}
    Let \(A\) be a ring and let \(\mfraka\) be a finitely generated ideal of \(A\).
    If \(A\) is derived \(\mfraka\)-complete, then \(A\) is \(\mfraka\)-Zariskian, that is, \(\mfraka \subseteq \Jac(A)\) where \(\Jac(A)\) is the Jacobson radical of \(A\).
\end{lemma}

Next, we recall the definition of prisms.

\begin{definition}[{\cite[Definition 3.2]{bhatt2022Prismsa}}] \label{DefPrism}
    Let \((A, I)\) be a pair of a \(\delta\)-ring \(A\) and its ideal \(I\).
    We call such a pair \((A, I)\) a \emph{\(\delta\)-pair}.\footnote{In the recent study \cite{antieau2023Prismatic} of a generalization of prismatic cohomology, a pair \((A, R)\) consisting a \(\delta\)-ring \(A\) and an \(A\)-algebra \(R\) is also called a \emph{\(\delta\)-pair} (\Cref{DefPrePrismatic}). Both terms are used in this paper because there is seldom any confusion (see \Cref{RemarkDeltaPair}).}
    The pair \((A, I)\) is a \textit{prism} if the following conditions hold:
    \begin{enumerate}
        \item \(I\) defines a locally free \(A\)-module of rank \(1\).
        \item \(A\) is derived \((p, I)\)-complete.
        \item \(p \in I + \varphi(I)A\).
    \end{enumerate}

    A prism \((A, I)\) is called
    \begin{enumerate}
        \item \textit{perfect} if \(A\) is a perfect \(\delta\)-ring, i.e., \(\varphi\) is an automorphism of \(A\).
        \item \textit{bounded} if \(A/I\) has bounded \(p^\infty\)-torsion.
        \item \textit{orientable} if \(I\) is a principal ideal of \(A\) and its generator is called an \textit{orientation} of \((A, I)\).
        \item \textit{crystalline} if \(I = (p)\).
    \end{enumerate}
    Note that any orientation \(\xi\) of an orientable prism \((A, I)\) becomes a non-zero-divisor and a distinguished element of \(A\) because of \cite[Lemma 2.25]{bhatt2022Prismsa}.
\end{definition}

\begin{remark} \label{RemarkBoundedAssumption}
    The assumption of being a bounded prism is foundational for several reasons:
    In their work \cite{bhatt2022Prismsa}, prisms are introduced as a form of ``deperfection'' in comparison to perfectoid rings.
    This assertion is grounded in their theorem establishing an equivalence between the category of perfect prisms and the category of perfectoid rings.
    Moreover, it is established that any perfect prism is, in fact, bounded.
    Consequently, when considering a perfectoid ring, a natural consideration arises for it to be a bounded prism.

    Another case is when \((A, I)\) is a prism with a Noetherian ring \(A\).
    It is also bounded because of the Noetherian assumption.
    Thus, the assumption of a prism being bounded encompasses both the former ``arithmetic'' case and the latter ``ring-theoretic'' case.
    Refer also to the accompanying remark (\Cref{RemarkPrismAssumption}) for further clarification.
\end{remark}

In commutative algebra in mixed characteristic, we need the following crystalline prism \((C(k), pC(k))\) which is well-known to the experts.

\begin{lemma} \label{CohenRing}
    Let \(k\) be a (not necessarily perfect) field of characteristic \(p\).
    Then there exists the Cohen ring \(C(k)\) of \(k\), that is, \(C(k)\) is the unique (up to isomorphism) absolutely unramified complete discrete valuation ring \(C(k)\) such that \(C(k)/p C(k)\) is isomorphic to \(k\).
    Furthermore, \(C(k)\) has a (not necessarily unique) \(\delta\)-structure consisting of a crystalline prism \((C(k), p C(k))\).
\end{lemma}

\begin{proof}
    The existence of the Cohen ring \(C(k)\) is well-known.
    The existence of a \(\delta\)-structure is a general theory of Cohen rings: By \cite[Theorem 29.2]{matsumura1986Commutative}, there exists a (non-unique) local homomorphism \(\map{\varphi}{C(k)}{C(k)}\) which induces the Frobenius map \(\map{F}{k}{k}\) on the residue fields (another reference is \cite[\S 1.2.4 and \S 3.3.1]{fontaineTheory}).
\end{proof}

Another important class of prisms is \emph{transversal prisms} (\Cref{DefofTransversal}) introduced by Ansch\"utz--Le Bras in \cite{anschutz2020Pcompleted}.
To define this, we need the following lemmas which guarantee some flexibility of regular sequences.

\begin{lemma} \label{PermutableDerivedComplete}
    Let \(A\) be a ring and \(x, y\) be a regular sequence on \(A\).
    If \(A\) is derived \(x\)-complete, then \(y, x\) is also a regular sequence on \(A\).
\end{lemma}

\begin{proof}
    Since \(A\) is \(x\)-torsion-free, \(A\) is \(x\)-adically complete by \cite[Lemma III.2.4]{bhatt2018Geometric} and, in particular, \(A\) is \(x\)-adically separated.
    By \cite[Corollary 7.8.8 (i)]{gabber2022Almost}, the sequence \((y, x)\) is a regular sequence on \(A\).
\end{proof}

\begin{lemma}[{\cite[Corollary 7.8.8 (ii)]{gabber2022Almost} and \citeSta{07DV}}] \label{PowerRegularSequence}
    Let \(A\) be a ring and let \(f_1, \dots, f_r\) be a sequence of elements of \(A\).
    Let \(e_1, \dots, e_r\) be positive integers.
    Then \(f_1, \dots, f_r\) is a regular sequence (resp., Koszul-regular sequence) on \(A\) if and only if \(f_1^{e_1}, \dots, f_r^{e_r}\) is a regular sequence (resp., Koszul-regular sequence) on \(A\).
\end{lemma}

The following lemma is similar to \cite[Lemma 3.3]{anschutz2020Pcompleted} and \cite[Lemma 2.7]{anschutz2023Prismatic}.

\begin{lemma} \label{EquivRegularSeq}
    Let $(A,I)$ be a $\delta$-pair. 
    Suppose that $A$ is derived $(p, I)$-complete and $I$ is generated by a distinguished element $d$.
    Then the following conditions are equivalent. 
\begin{enumerate}
    \item $p,d$ is a regular sequence on \(A\).
    \item $d,p$ is a regular sequence on \(A\). 
    \item $\varphi(d), d$ is a regular sequence on \(A\). 
    \item $d, \varphi(d)$ is a regular sequence on \(A\). 
    \item $p, \varphi(d)$ is a regular sequence on \(A\). 
    \item $\varphi(d), p$ is a regular sequence on \(A\). 
    \item $p, \varphi^{i+1}(d)$ is a regular sequence for every $i\geq 0$ on \(A\).
    \item $\varphi^{i+1}(d), p$ is a regular sequence for every $i\geq 0$ on \(A\). 
    \item $\varphi^{i+1}(d), \varphi^{j+1}(d)$ is a regular sequence for every $i,j \geq 0$, $i \neq j$ on \(A\).
\end{enumerate}

    If one of the equivalent conditions is satisfied, $A$ is $(p, I)$-adically complete.
\end{lemma}

\begin{proof}
    We show the following implications:
    \begin{center}
        \begin{tikzcd}
            (6)    & (5) \arrow[l, Leftrightarrow] & (9) \arrow[rd, Rightarrow] &                                    \\
            (7) \arrow[d, Leftrightarrow] & (1) \arrow[r, Leftrightarrow] \arrow[u, Leftrightarrow] \arrow[l, Rightarrow] \arrow[ru, Rightarrow] & (2) \arrow[r, Leftrightarrow]  & (4)                                \\
            (8) \arrow[rrr, Rightarrow]    & & & (3). \arrow[u, Leftrightarrow]
            \end{tikzcd}
    \end{center}

    Since $A$ is derived $(p, I)$-complete and $\varphi^{i}(d) \in (p, I)A$ for every $i \geq 0$, $A$ is derived $p$-complete and derived $\varphi^i(d)$-complete by definition (see \cite[Definition III.2.1]{bhatt2018Geometric}). 
    Hence each one of the above regular sequences is permutable by \Cref{PermutableDerivedComplete} and thus the equivalences \IFF{(1)}{(2)}, \IFF{(3)}{(4)}, \IFF{(5)}{(6)}, and \IFF{(7)}{(8)} hold.
    Moreover, because of $\varphi(d)=d^{p}+p\delta(d)$ with an invertible element $\delta(d)$, we have the equivalences \IFF{(2)}{(4)} and \IFF{(1)}{(5)} by \Cref{PowerRegularSequence}.
    Consequently, the conditions $(1)$-$(6)$ are equivalent.

    Since the derived \((p, I)\)-complete ring \(A\) is \((p, I)\)-Zariskian by \Cref{LemmaDerivedComplete}, \(\varphi^{i+1}(d)\) is a distinguished element for any \(i \geq 0\) by \cite[Lemma 2.25]{bhatt2022Prismsa}.
    So we can apply the implication \IMPLIES{(1)}{(5)} for a sequence \(p, \varphi^i(d)\) for each \(i \geq 0\) and then we have \IMPLIES{(1)}{(7)}.
    By \cite[Lemma 2.7]{anschutz2023Prismatic}, we have \IMPLIES{(1)}{(9)}.
    The immediate implications \IMPLIES{(8)}{(3)} and \IMPLIES{(9)}{(4)} conclude the proof.

    If $p,d$ is a regular sequence on $A$, by the paragraph above \cite[Lemma 3.3]{anschutz2020Pcompleted}, this $A$ is $(p,I)$-adically complete.
\end{proof}

\begin{definition}[{cf. \cite[Definition 3.2]{anschutz2020Pcompleted}}] \label{DefofTransversal}
    We say that an orientable\footnote{In \cite[Definition 2.1.3]{bhatt2022Absolute}, we can define the notion of a \emph{transversal prism} for any prism (not necessarily orientable).
    However, we assume that transversal prisms are orientable in this paper.} prism $(A, I)$ is \emph{transversal} if some (or, equivalently, any) orientation $\xi \in I$ forms a regular sequence \(p, \xi\) of $A$ or satisfies one of the equivalent conditions (\Cref{EquivRegularSeq}).  
\end{definition}

Transversal prisms are well-behaved in the following sense (\Cref{TransversalFFlatEquiv}).
We first recall the notion of a \emph{completely faithfully flat} map.

\begin{definition}[{cf. \cite{bhatt2022Prismsa,yekutieli2018Flatness}}] \label{DefofCompFlat}
    Let \(A\) be a ring and \(M\) be an \(A\)-module.
    Fix an ideal \(\mfraka\) of \(A\).
    Then \(M\) is an \emph{\(\mfraka\)-completely flat}\footnote{The original definition of the \(\mfraka\)-completely flatness is that \(M \otimes^L_A N\) is concentrated in degree \(0\) for any \(A/\mfraka\)-module \(N\). These equivalences are shown in \cite[Theorem 4.3]{yekutieli2018Flatness}.} \(A\)-module if the derived tensor product \(M \otimes^L_A A/\mfraka\) is concentrated in degree \(0\) and \(M/\mfraka M\) is a flat \(A/\mfraka\)-module.
    Moreover if \(M/\mfraka M\) is a faithfully flat \(A/\mfraka\)-module, we say that \(M\) is \emph{\(\mfraka\)-completely faithfully flat}.

    A map of rings \(A \to B\) is \emph{\(\mfraka\)-completely (faithfully) flat} if \(B\) is an \(\mfraka\)-completely (faithfully) flat \(A\)-module via this map \(A \to B\).
\end{definition}

If the base ring is Noetherian, the situation is simpler.
To use this fact later (\Cref{NoetherianCase} and \Cref{FFlatGivesRegular}), we formulate the following lemma in a form suitable for our setup:

\begin{lemma} \label{CompletionFaithfullyFlat}
    Let \(R\) be a commutative ring with bounded \(\pi^\infty\)-torsion for some element \(\pi\) in \(R\) and let \(S\) be a \(\pi\)-adically complete \(R\)-algebra.
    If \(R/\pi^nR \to S/\pi^nS\) is flat (resp., faithfully flat) for all \(n \geq 1\), then \(S\) becomes a \(\pi\)-completely flat (resp., \(\pi\)-completely faithfully flat) \(R\)-algebra.
    If \(R\) is furthermore Noetherian (resp., \(\pi\)-Zariskian and Noetherian), then \(S\) is flat (resp., faithfully flat) over \(R\).
\end{lemma}

\begin{proof}
    Since \(R\) has bounded \(\pi^\infty\)-torsion, \(\pi\) is a weakly proregular element\footnote{An element \(a\) of a commutative ring \(R\) is called \emph{weakly proregular} if the inverse system \(\{H^q(\Kos(R; a^i))\}_{i \geq 0}\) of \(R\)-modules is pro-zero for every \(q < 0\). Here, \(H^q(\Kos(R; a^i))\) is the \(q\)-th cohomology of the Koszul complex and its transition morphism \(\Kos(R; a^{i+1}) \to \Kos(R; a^i)\) is given by the multiplication map \(\times a\) in degree \(-1\) and the identity map in degree \(0\). See \cite{yekutieli2021Weak} for more details.
} of \(R\) by \cite[Proposition 5.6]{yekutieli2021Weak}.
    So \(S = \lim_n S/\pi^nS\) is \(\pi\)-completely flat over \(R\) by \cite[Theorem 1.6 (1)]{yekutieli2018Flatness}.
    If \(R/\pi^n R \to S/\pi^n S\) is faithfully flat (in particular, \(n = 1\)), \(S\) is moreover \(\pi\)-completely faithfully flat over \(R\) by the definition of \(\pi\)-completely faithful flatness.

    By using this, if \(R\) is Noetherian and \(R/\pi^n R \to S/\pi^n S\) is flat, then \(S\) is flat over \(R\) by \cite[Theorem 1.6 (2)]{yekutieli2018Flatness}.
    Assume moreover that \(R\) is a \(\pi\)-Zariskian Noetherian ring and \(R/\pi^n R \to S/\pi^n S\) is faithfully flat for all \(n \geq 1\).
    To show that \(S\) is faithfully flat over \(R\), it suffices to show that \(\Spec(S) \to \Spec(R)\) is surjective. We follow the argument in the proof of \cite[Proposition 5.1]{bhatt2018Direct}:
    Since \(R \to S\) is flat, it satisfies the going down property. Thus, it suffices to show that the fiber \(R/\mfrakm \otimes_R S\) of a maximal ideal \(\mfrakm\) of \(R\) is non-zero.
    Since \(R\) is \(\pi\)-Zariskian, this \(\mfrakm\) contains \(\pi\) and hence this fiber is isomorphic to \(R/\mfrakm \otimes_{R/\pi R} S/\pi S\), which is non-zero because \(R/\pi R \to S/\pi S\) is faithfully flat.
\end{proof}

We end this section with the following lemma, which gives the equivalence condition for the completely faithful flatness of the Frobenius lift.

\begin{lemma} \label{TransversalFFlatEquiv}
    Let \((A, I)\) be a transversal or crystalline prism, and let \(\map{\varphi}{A}{\varphi_* A}\) be the Frobenius lift of the \(\delta\)-ring \(A\).
    Consider the following conditions:
    \begin{enumerate}
        \item \(\varphi\) is \(p\)-completely faithfully flat.
        \item \(\varphi\) is \(I\)-completely faithfully flat.
        \item \(\varphi\) is \((p, I)\)-completely faithfully flat.
        \item The Frobenius map \(\map{F}{A/(p)}{F_*(A/(p))}\) is faithfully flat.
        \item The map \(\map{\overline{\varphi}_I}{A/I}{\varphi_*(A/\varphi(I))}\) induced from \(\varphi\) is faithfully flat.
        \item The \(p\)-th power map \(\map{\overline{\varphi}_{(p, I)}}{A/(p, I)}{\varphi_*(A/(p, I^{[p]}))}\) induced from \(\varphi\) is faithfully flat, where \(I^{[p]}\) is the ideal generated by \(\set{f^p \in A}{f \in I}\) in \(A\), which is called the \emph{Frobenius power} of \(I\).
    \end{enumerate}
    Then we have \IFF{(1)}{(4)}, \IFF{(2)}{(5)}, and \IFF{(3)}{(6)}.
    If moreover \(A\) is Noetherian, then all of the above conditions are equivalent.
\end{lemma}

\begin{proof}
    The implications \IMPLIES{(1)}{(4)}, \IMPLIES{(2)}{(5)}, \IMPLIES{(3)}{(6)}, and \IMPLIES{((4) or (5))}{(6)} are clear.
    It suffices to show that \IMPLIES{(4)}{(1)}, \IMPLIES{(5)}{(2)}, and \IMPLIES{(6)}{(3)}.
    Note that the transversal or crystalline assumption is only used in the proof of them.

    Since the proofs are similar, we only show \IMPLIES{(6)}{(3)}.
    By the definition of \((p, I)\)-completely faithful flatness (\Cref{DefofCompFlat}), it is sufficient to show that the derived tensor product \(A/(p, I) \otimes^L_A \varphi_*A\) is concentrated in degree \(0\).

    If \((A, I)\) is transversal, a fixed orientation \(\xi\) of \((A, I)\) gives a regular sequence \(p, \xi\) on \(A\).
    We have a projective resolution of \(A/(p)\) (resp., \(A/(p, \xi)\)):
    \begin{align*}
        0 \to A & \xrightarrow{\times p} A \to A/(p) \to 0 \\
        (\text{resp.,}\  0 \to A/(p) & \xrightarrow{\times \xi} A/(p) \to A/(p, \xi) \to 0)
    \end{align*}
    as an \(A\)-module (resp., an \(A/(p)\)-module).
    Note that there is a canonical isomorphism of derived tensor products
    \begin{equation*}
        A/(p, I) \otimes^L_{A} \varphi_* A \cong A/(p, I) \otimes^L_{A/(p)} (A/(p) \otimes^L_{A} \varphi_* A)
    \end{equation*}
    in \(D(A/(p))\) by \citeSta{06Y6}.
    First, we have a (quasi-)isomorphism in \(D(A)\);
    \begin{align*}
        A/(p) \otimes^L_{A} \varphi_* A \cong & (0 \to A \otimes_{A} \varphi_* A \xrightarrow{\times p \otimes \id_{\varphi_* A}} A \otimes_{A} \varphi_* A \to 0) \\
        \cong & (0 \to \varphi_* A \xrightarrow{\times \varphi(p)} \varphi_* A \to 0)
    \end{align*}
    Since \(A\) is \(p\)-torsion-free, this complex \(A/(p) \otimes^L_A \varphi_* A\) is concentrated in degree \(0\) and isomorphic to \(\varphi_*(A/(p))\) in \(D(A)\).
    Second, we have a (quasi-)isomorphism in \(D(A/(p))\);
    \begin{align*}
        A/(p, \xi) \otimes^L_{A/(p)} (A/(p) \otimes^L_{A} \varphi_* A) & \\
        \cong (0 \to A/(p) \otimes_{A/(p)} (\varphi_*(A/(p))) & \xrightarrow{\times \xi \otimes \id_{\varphi_*(A/(p))}} A/(p) \otimes_{A/(p)} (\varphi_*(A/(p))) \to 0)\\
        \cong (0 \to \varphi_*(A/(p)) & \xrightarrow{\times \varphi(\xi)} \varphi_*(A/(p)) \to 0)
    \end{align*}
    Similarly as above, this complex is concentrated in degree \(0\) and isomorphic to \(\varphi_*(A/(p, \xi^p))\) in \(D(A/(p))\) and we are done.

    If \((A, I)\) is crystalline, \(I = (p)\) and thus \(A\) is \(p\)-torsion-free.
    So we can show that \(A/(p) \otimes^L_A \varphi_* A\) is concentrated in degree \(0\) as above.

    If \(A\) is Noetherian, it suffices to show that \IMPLIES{(3)}{(1)} and \IMPLIES{(3)}{(2)} but this follows from \cite[Proposition 5.1]{bhatt2018Direct}.
\end{proof}

\section{Regular Prisms} \label{SectionRegularPrisms}

In this section, we define the notion of \emph{regular (local) prisms} and solve a deformation problem of the regularity of \(\delta\)-rings (\Cref{EquivRegularPrism}).
This is a fundamental object when applying the theory of prismatic complexes to commutative ring theory.

For the sake of generality, we collect some lemmas of \(\delta\)-rings.

\begin{lemma}[{cf. \cite[\S 2.5 Exercises 4]{kedlayaNotes} and \cite[Lemma 2.28]{bhatt2022Prismsa}}] \label{pPowerTorsionNil}
    Let \(A\) be a \(\delta\)-ring.
    Then we have the following:
    \begin{enumerate}
        \item If an element \(a\) of \(A\) is \(p^n a = 0\) for some integer \(n > 0\),
        then we have \(p^{n+1} \delta(a) = 0\) in \(A\).
        \item The submodule \(A[p^\infty]\) of \(p^\infty\)-torsion elements of \(A\) is contained in the nilradical \(\Nil(A)\).
    \end{enumerate}
\end{lemma}

\begin{proof}
    (1): 
    We have
    \begin{equation*}
        p^{n+1} \delta(a) = p^n (\varphi(a) - a^p) = \varphi(p^n a) - p^n a^p = 0
    \end{equation*}
    and this is what we want.

    (2): If \(a \in A\) is in the submodule \(A[p]\) of \(p\)-torsion elements of \(A\), we have \(p^2 \delta(a) = 0\) by (1).
    Applying \(\delta\) for \(pa = 0\), we have \(a^p + p\delta(a) = \varphi(a) = 0\) as in the proof of \cite[Lemma 2.28 (1)]{bhatt2022Prismsa}.
    Therefore, \(a^{p^2} = (-1)^p p^p \delta(a)^p = 0\) and thus \(a\) is in \(\Nil(A)\).
    By induction, we can assume that the submodule \(A[p^{n-1}] \subset \Nil(A)\) of \(p^{n-1}\)-torsion elements of \(A\) for an integer \(n \geq 2\).
    For any \(a \in A[p^n]\), we have
    \begin{equation*}
        0 = \delta(0) = \delta(p^n a) = \varphi(p^n) \delta(a) + a^p \delta(p^n) = p^n \delta(a) + a^p p^{n-1} (1 - p^{np - n}).
    \end{equation*}
    Since we only consider \(\setZ_{(p)}\)-algebras, \(1 - p^{np - n}\) is an invertible element of \(A\).
    So we have
    \begin{equation*}
        p^{n-1} \parenlr{a^p + \frac{p \delta(a)}{1 - p^{np - n}}} = 0
    \end{equation*}
    and thus \(a^p + p (\delta(a)/(1 - p^{np - n}))\) is in \(\Nil(A)\) by induction hypotheses.
    By our assumption of \(a\), (1) implies that \(p^{n+1}\delta(a) = 0\) and thus \(p \delta(a)\) is also in \(\Nil(A)\).
    Therefore, \(a^p\) is in \(\Nil(A)\).
\end{proof}

\begin{lemma} \label{pTorsionFreePrime}
    Let \(A\) be a \(p\)-adically separated \(\delta\)-ring.
    If \(p A\) is a prime ideal of \(A\), then \(A\) is \(p\)-torsion-free.
\end{lemma}

\begin{proof}
    Take an element \(a \in A\) such that \(pa = 0\).
    Since any \(p^\infty\)-torsion \(A[p^\infty]\) of \(\delta\)-ring \(A\) is contained in the nilradical \(\Nil(A)\) by \Cref{pPowerTorsionNil},
    there exists an integer \(N > 0\) such that \(a^N = 0 \in p A\).
    Since \(p A\) is a prime ideal of \(A\), there exists an element \(a_1 \in A\) such that \(a = p a_1 \in p A\).
    Because of \(0 = pa = p^2 a_1\), there exists an element \(a_2 \in A\) such that \(a_1 = p a_2\) by the same reason.
    Repeating this process, we have
    \begin{equation*}
        A[p] \subseteq \bigcap_{n \geq 0} p^n A = 0
    \end{equation*}
    by the \(p\)-adically separatedness of \(A\) and we finish the proof.
\end{proof}

The following lemma shows that the existence of a \(\delta\)-structure of a local ring \(A\) gives a restriction of its ring structure.
This is a refinement of \cite[Notation 9.3]{bhatt2022Prismatization}. Further properties of \(\delta\)-structures are also shown in \cite[Remark 2.4]{hochster2024Jacobia}.

\begin{lemma} \label{NoDeltaStructure}
    Let \(A\) be a \(p\)-Zariskian ring (not necessarily Noetherian) with the Jacobson radical \(\Jac(A)\).
    If \(A\) has a \(\delta\)-structure, then
    \begin{equation*}
        \delta(\Jac(A)^2) \subseteq \Jac(A).
    \end{equation*}
    In particular, any distinguished element \(d\) does not belong to \(\Jac(A)^2\).
    If the \(\delta\)-ring \(A\) is local with the unique maximal ideal \(\mfrakm\), then \(A\) is unramified, namely, \(p \in \mfrakm \setminus \mfrakm^2\).
\end{lemma}

\begin{proof}
    Let \(x\) be an element of \(\Jac(A)^2\).
    Let \(\delta\) be a \(\delta\)-structure of \(A\) and let \(f_i\) and \(g_i\) be elements of \(\Jac(A)\) such that \(x = \sum_{i=1}^n f_i g_i\) in \(A\) for some \(n \geq 1\).
    We show that \(\delta(x)\) is in \(\Jac(A)\) by induction of \(n \geq 1\).
    In the case of \(n = 1\), we have
    \begin{equation*}
        \delta(x) = \delta(f_1 g_1) = f_1^p \delta(g_1) + g_1^p \delta(f_1) + p \delta(f_1)\delta(g_1) \in (p, f_1, g_1)A \subseteq \Jac(A)
    \end{equation*}
    since \(A\) is \(p\)-Zariskian and thus \(p \in \Jac(A)\).
    Suppose the claim holds for \(n \geq 1\).
    We can write \(x = f_1 g_1 + y\) where \(y = \sum_{i=2}^n f_i g_i\) and then we have
    \begin{equation*}
        \delta(x) = \delta(f_1 g_1) + \delta(y) + \frac{((f_1 g_1)^p + y^p) - (f_1 g_1 + y)^p}{p} \in \Jac(A).
    \end{equation*}

    Since \(d\) is a distinguished element of \(A\), \(\delta(d)\) is a unit element of \(A\).
    If \(d\) is contained in \(\Jac(A)^2\), the above claim shows \(A^\times \ni \delta(d) \in \Jac(A)\).
    This is a contradiction.
    By a direct calculation or \cite[Example 2.6]{bhatt2022Prismsa}, \(p\) is a distinguished element in the \(p\)-Zariskian \(\delta\)-ring \(A\).
    So if \(A\) is local, it is unramified.
\end{proof}

Our first main result is to prove the deformation problem of the regularity of \(\delta\)-rings as follows.

\begin{proposition} \label{EquivRegularPrism}
    Let \((A, I)\) be a bounded prism (here we do not yet assume that \(A\) is Noetherian).
    Then the following conditions are equivalent:
    \begin{enumerate}
        \item[(1)] \(A\) is a regular ring\footnote{Here, a \emph{regular ring \(R\)} is a Noetherian ring \(R\) whose localization \(R_\mfrakm\) is a regular local ring for any maximal ideal \(\mfrakm\). In this paper, if any \(R_\mfrakm\) is unramified, we say that \(R\) is an \emph{unramified regular ring}.} (resp., regular local ring).
        \item[(1')] \(A\) is an unramified regular ring (resp., unramified regular local ring).
        \item[(2)] \(A/p A\) is a regular ring (resp., regular local ring).
        \item[(3)] \(A/I\) is a regular ring (resp., regular local ring).
    \end{enumerate}

    If one of the equivalent conditions is satisfied, \((A, I)\) is a transversal prism or a crystalline prism.
\end{proposition}

\begin{proof}
    Since \(A\) is derived \((p, I)\)-complete, \(A\) is \((p, I)\)-Zariskian by \Cref{LemmaDerivedComplete}.
    So any maximal ideal \(\mfrakm\) of \(R\) contains \(p\) and \(I\).
    Thus we show the case of local rings.

    In any case (1), (1'), (2), and (3), \(A\) is a local ring, and thus the Picard group of \(A\) is trivial.
    So \(I\) is free of rank \(1\) and we can fix an orientation \(\xi\) of \(I\) which is a non-zero-divisor and a distinguished element of \(A\) by \cite[Lemma 2.25]{bhatt2022Prismsa}.

    (1) \(\Leftrightarrow\) (1'): By \Cref{NoDeltaStructure}, this is clear.

    (1') \(\Rightarrow\) (2): In any unramified regular local ring, \(p\) can be extended to a regular system of parameters. Then \(A/pA\) is a regular local ring by \cite[Proposition 2.2.4]{bruns1998CohenMacaulay}.

    (2) \(\Rightarrow\) (1'): Since \(A\) is \(p\)-adically separated and the regular local ring \(A/pA\) is a Noetherian integral domain, \(A\) is also a Noetherian local \(p\)-torsion-free ring by \citeSta{05GH} and \Cref{pTorsionFreePrime}.
    By \citeSta{00NU}, \(A\) is a regular local ring with a \(\delta\)-structure.
    Then \(A\) is an unramified regular local ring by \Cref{NoDeltaStructure}.

    (3) \(\Rightarrow\) (1): Since \(A\) is \(I\)-adically complete and \(A/I\) is a Noetherian local ring, \(A\) is also a Noetherian local ring as above.
    Since \(I\) is generated by a (length \(1\)) regular sequence on \(A\) and \(A/I\) is a regular local ring, \(A\) is also a regular local ring as above.

    (1) \(\Rightarrow\) (3): Since \(A\) is a regular local ring, it suffices to show that \(\xi\) can be extended to a regular system of parameters of \(A\) by \cite[Proposition 2.2.4]{bruns1998CohenMacaulay}.
    Since \(\xi\) is a distinguished element of the \(\delta\)-ring \(A\), \(\xi\) belongs to \(\mfrakm \setminus \mfrakm^2\) by \Cref{NoDeltaStructure} and this shows the claim.

    If \(\xi = p\), \((A, I)\) is a crystalline prism.
    Assume that \(\xi \neq p\).
    Since \(\xi\) is already a non-zero-divisor of \(A\) and \(A/I\) is a regular local ring, the sequence \(\xi, p\) is a regular sequence on \(A\).
    Then \((A, I)\) is a transversal prism.
\end{proof}

\begin{remark} \label{RemarkImplicationsEquivRegularPrism}
    The assumption that \((A, I)\) is a bounded prism in \Cref{EquivRegularPrism} is somewhat redundant.
    Let \((A, I)\) be a pair of a ring \(A\) (not necessarily admitting a \(\delta\)-structure) and its ideal \(I\).
    To prove the equivalence in \Cref{EquivRegularPrism}, we only use the existence of a bounded prism structure of the pair \((A_\mfrakm, I_\mfrakm)\) for any maximal ideal \(\mfrakm\) even if \((A, I)\) is not a prism.
    
    Furthermore, we only use the following assumptions respectively:    %
    \begin{itemize}
        \item \IFF{(1)}{(1')}: \(A_\mfrakm\) has a \(\delta\)-structure.
        \item \IMPLIES{(1')}{(2)}: This is true for any unramified regular local ring \(A_\mfrakm\) of residue characteristic \(p\).
        \item \IMPLIES{(2)}{(1')}: \(A_\mfrakm\) has a \(\delta\)-structure and is \(p\)-adically complete.
        \item \IMPLIES{(3)}{(1)}: \(A_\mfrakm\) is \(I\)-adically complete and \(I_\mfrakm\) is generated by a regular sequence on \(A_\mfrakm\).
        \item \IMPLIES{(1)}{(3)}: \(A_\mfrakm\) has a \(\delta\)-structure and \(I_\mfrakm\) is a principal ideal generated by a distinguished element of \(A_\mfrakm\).
    \end{itemize}
\end{remark}

\begin{definition} \label{DefofRegularPrism}
    Let \((A, I)\) be a transversal or crystalline prism.\footnote{By \Cref{EquivRegularPrism} and \Cref{RemarkImplicationsEquivRegularPrism}, if \((A, I)\) is only assumed to be a \((p, I)\)-adically complete prism, we can define a \emph{regular prism} \((A, I)\) in the same way. However, any such a regular prism is automatically transversal or crystalline. So the transversal or crystalline assumption does not lose its generality.}
    We say that \((A, I)\) is a \emph{regular (local) prism} if it satisfies one of the equivalent conditions of \Cref{EquivRegularPrism}.
    Furthermore, if \(A\) is a complete local ring, we will say that \((A, I)\) is a \emph{complete regular local prism}.
    Note that a regular local prism \((A, I)\) does not necessarily make a complete local ring \(A\) although \(A\) is \((p, I)\)-adically complete because of boundedness of \((A, I)\).

\end{definition}

\begin{example}
    There exists an orientable prism \((A, I)\) such that \(A\) is a Noetherian ring but not a regular local ring (see \cite[Remark 2.29]{bhatt2022Prismsa}).
    Even if \(A\) is a Noetherian \emph{domain}, such an example exists (see \Cref{LogRegularPrism}).
\end{example}

The above \Cref{EquivRegularPrism} shows the following prismatic interpretation of Cohen's structure theorem.
We conclude that the prismatic structures on an unramified complete regular local ring serve as a ``moduli space'' of regular local rings of residue characteristic \(p\).
This situation presents a commutative algebraic analogue, akin to the case where the prismatic structures on a Witt ring define the \emph{Fargues--Fontaine curve}, which is a ``moduli space'' of untilts of perfectoid field of characteristic \(p\) (see \cite[\S 2.3 p.15]{bhatt2019Perfectoid}).

\begin{corollary} \label{EquivRegularRing} \label{CorrespCompleteRegularLocalRing}
    Fix a (not necessarily perfect) field \(k\) of positive characteristic \(p\) and take the Cohen ring \(C(k)\) of \(k\) equipped with a \(\delta\)-structure (see \Cref{CohenRing}).
    Fix an integer \(d \geq 0\).
    Set a formal power series ring \(A \defeq A(k,d) \defeq C(k)[|T_1, \dots, T_d|]\) equipped with a \(\delta\)-structure given by \(\varphi(T_i) = T_i^p\).
    Then we have a one-to-one correspondence between the following two sets:
    \[
    \begin{gathered}
      \mscrR \defeq \{\text{complete regular local ring \(R\) of dimension \(d\) with residue field \(k\)}\}/\cong \\
      \updownarrow \\
      \mscrI \defeq \set{I \subseteq A}{I~\text{is an ideal of}~A~\text{forming a prism}~(A, I)}/\sim,
    \end{gathered}
    \]
    where the upper equivalence relation \(\cong\) is the isomorphism classes of rings and the lower equivalence relation \(\sim\) is defined by \(I \sim I'\) if and only if there exists an automorphism \(f\) of a ring \(A\) (not necessarily an automorphism of a \(\delta\)-ring) such that \(f(I') = I\).
\end{corollary}

\begin{proof}
    Since \(A\) is an (unramified) complete regular local ring, an ideal \(I\) of \(A\) forming a prism \((A, I)\) makes a complete regular local prism \((A, I)\) by \Cref{EquivRegularPrism}.
    So the operation
    \begin{align*}
        \mscrI & \longrightarrow \mscrR, \\
        I & \longmapsto A/I
    \end{align*}
    defines a map of sets.

    (Surjectivity): Let \(R\) be a complete regular local ring of dimension \(d\) with residue field \(k\).
    Fix a regular system of parameters \(x_1, \dots, x_d\) of \(R\).
    By Cohen's structure theorem, there exists a surjective map of rings \(A \twoheadrightarrow R\) such that \(T_i \mapsto x_i\).
    As in \Cref{Envelope} below, the kernel \(I \defeq \ker(A \twoheadrightarrow R)\) forms a prism \((A, I)\) because of \(\dim(A/I) = \dim(A) - 1 = d = \dim(R)\).
    This shows the surjectivity.
    
    (Injectivity): Let \(I\) and \(I'\) be ideals of \(A\) forming prisms \((A, I)\) and \((A, I')\) respectively.
    Assume that \(R \defeq A/I \cong A/I'\).
    It suffices to show that there exists \(f \in \Aut(A)\) such that \(f(I') = I\).
    We define two surjective maps:
    \begin{align*}
        \pi & \colon A \twoheadrightarrow A/I = R, \quad T_i \mapsto x_i, \\
        \pi' & \colon A \twoheadrightarrow A/I' \cong R, \quad T_i \mapsto x_i'.
    \end{align*}
    Since \(\mfrakm_R = (x_1, \dots, x_d)R = (x_1', \dots, x_d')R\), there exists a \(d \times d\)-matrix \(M = (m_{ij})\) over \(A\) such that
    \begin{equation*}
        x_i' = \sum_{k=1}^d \pi(m_{ik}) x_k \in R
    \end{equation*}
    for \(1 \leq i \leq d\).
    Considering the cotangent space \(\mfrakm_R/\mfrakm_R^2 = k x_1 \oplus \dots \oplus k x_d = k x'_1 \oplus \dots \oplus k x'_d\) of \(R\), the image of \(M\) in the set \(M_d(k)\) of \(d \times d\)-matrices over \(k\) via \(A \xrightarrow{\pi} R = A/I \xrightarrow{x_i \mapsto 0} C(k) \twoheadrightarrow C(k)/pC(k) \cong k\) is a change-of-basis matrix on \(\mfrakm_R/\mfrakm_R^2\) and thus invertible in \(M_d(k)\).
    In particular, \(M\) is also invertible in \(M_d(A)\) since \(A\) is local.
    This gives an automorphism \(f \in \Aut(A)\) such that
    \begin{equation*}
        f(T_i) \defeq \sum_{k=1}^d m_{ik} T_k \in A
    \end{equation*}
    for \(1 \leq i \leq d\).
    To prove \(f(I') = I\), it suffices to show that \(f\) induces a commutative diagram:
    \begin{center}
        \begin{tikzcd}
            A \arrow[d, "f"] \arrow[d, "\cong"'] \arrow[r, "\pi'", two heads] & R \arrow[d, Rightarrow, no head] \\
            A \arrow[r, "\pi", two heads]                                     & R.                               
        \end{tikzcd}
    \end{center}
    This is deduced from a direct calculation:
    \begin{equation*}
        \pi'(T_i) = x_i' = \sum_{k=1}^d \pi(m_{ik}) x_k = \pi\parenlr{\sum_{k=1}^d m_{ik} T_k} = \pi(f(T_i)).
    \end{equation*}
\end{proof}

\section{Faithful Flatness of Prismatic Complexes} \label{SectionPrismaticComplexes}

In this section, we first recall the definition of prismatic complexes and their properties based on \cite{bhatt2022Prismsa,bhatt2022Absolute}.
We also prove a main theorem (\Cref{AnimatedLCIPrism}) which is used in the proof of prismatic Kunz's theorem (\Cref{AnimatedCase})

To clarify our setting, we introduce the notion of a \emph{\(\delta\)-pair}, which was recently defined in \cite[\S 2]{antieau2023Prismatic}.

\begin{definition} \label{DefPrePrismatic}
    A pair \((A, R)\) of a \(\delta\)-ring \(A\) and an \(A\)-algebra \(R\) is called a \emph{\(\delta\)-pair}.
    If the structure map \(A \to R\) is surjective, \((A, R)\) is called a \emph{surjective \(\delta\)-pair}.
    \begin{itemize}
        \item A \(\delta\)-pair \((A, R)\) is \emph{pre-prismatic} if the kernel of the structure map \(A \to R\) contains a locally free ideal \(I\) of rank \(1\) and, Zariski locally on \(\Spec(A)\), \(I\) is generated by an element \(d\) such that \(\delta(d)\) is a unit in the \(p\)-adic completion \(\widehat{R}\) of \(R\).
        \item A pre-prismatic \(\delta\)-pair \((A, R)\) is a \emph{prismatic \(\delta\)-pair with a prism \((A, I)\)} if \(I\) is a locally free ideal of rank \(1\) contained in the kernel of \(A \to R\) and the \(\delta\)-pair\footnote{Here we use the notion of a \emph{\(\delta\)-pair} in the sense of \Cref{DefDeltaRing}. See \Cref{RemarkDeltaPair}.} \((A, I)\) becomes a prism.
    \end{itemize}
    A prismatic \(\delta\)-pair \((A, R)\) with a prism \((A, I)\) satisfying a property \(\mcalP\) is called a \emph{prismatic \(\delta\)-pair with a \(\mcalP\) prism \((A, I)\)}, for example, we have the notion of a \emph{prismatic \(\delta\)-pair with a perfect/bounded/orientable/crystalline/(complete) regular (local) prism \((A, I)\)}.
\end{definition}

\begin{remark} \label{RemarkDeltaPair}
    As mentioned in the footnote in \Cref{DefPrism}, there is a different term of a \emph{\(\delta\)-pair} \((A, I)\) consisting of a \(\delta\)-ring \(A\) and an ideal \(I\) of \(A\) in \cite{bhatt2022Prismsa}.
    By definition, the notion of a \(\delta\)-pair in this sense (resp., a prism) is equivalent to the notion of a surjective \(\delta\)-pair \((A, R)\) (resp., a surjective prismatic \(\delta\)-pair \((A, R)\) with a prism \((A, \ker(A \twoheadrightarrow R))\)).
    Thus the notion of a \(\delta\)-pair in the sense of \cite{antieau2023Prismatic} is a generalization of the notion of a \(\delta\)-pair in the sense of \cite{bhatt2022Prismsa}.
    For this reason, we use the term \emph{\(\delta\)-pair} in both senses.
\end{remark}

\begin{construction} \label{ConstructionAnimatedQuotient}
    Let \((A, I)\) be a bounded prism and let \(R \defeq A/J\), where \(J \defeq (I, f_1, \dots, f_r)\) with a sequence of elements \(f_1, \dots, f_r\) in \(A\) (not necessarily a Koszul-regular sequence).
    Then the pair \((A, R)\) becomes a surjective prismatic \(\delta\)-pair with a bounded prism \((A, I)\).
    Hereafter we use the formalism of \emph{animated rings}, which is developed in \cite{cesnavicius2024Purity,lurie2018Spectral}.
    We briefly review it in \Cref{AnimatedAppendix}.

    We define an animated \(A/I\)-algebra by taking a derived quotient (see \cite[\S 5.1.7]{cesnavicius2024Purity} or \cite[\S 2.3.1]{khan2019Virtual})
    \begin{equation*}
        R^{an}(\underline{f}) \defeq R^{an}(f_1, \dots, f_r) \defeq (A/I)/^L(\overline{f_1}, \dots, \overline{f_r}) \in \ACAlg_{A/I},
    \end{equation*}
    where each \(\overline{f_i}\) is the canonical image of \(f_i \in A\) in \(A/I\) and \(\ACAlg_{R/I}\) is the \(\infty\)-category of animated \(A/I\)-algebras (\Cref{AnimatedRings}).
    We will often omit the symbol \(\overline{(-)}\) and write \(\overline{f_i}\) as \(f_i\).
    The underlying complex of \(R^{an}(f_1, \dots, f_r)\) is isomorphic to the Koszul complex \(\Kos(A/I; f_1, \dots, f_r)\).\footnote{In derived algebraic geometry, the derived quotient \(A/I \to R^{an}(f_1, \dots, f_r)\) defines a \emph{Koszul derived scheme \(\mcalV(f_1, \dots, f_r)\)} and a \emph{quasi-smooth closed immersion} of derived schemes \(\Spec(R^{an}(f_1, \dots, f_r)) \to \Spec(A/I)\) (see, for example, \cite{kerz2018Algebraic, khan2019Virtual}).}
\end{construction}

\begin{warning} \label{WarningChoice}
    In general, due to the sensitivity of Koszul complexes \cite[Proposition 1.6.21]{bruns1998CohenMacaulay}, the (underlying Koszul complex of the) animated ring \(R^{an}(f_1, \dots, f_r)\) depends on the \emph{choice} of a generator \(f_1, \dots, f_r\) of \(J/I\) in \(A/I\).
    However, this animated ring is invariant under \emph{permutation} of \(f_1, \dots, f_r\) (not only as a complex) by the universal property of derived quotients (\cite[Lemma 2.3.5]{khan2019Virtual}). 
    
    Furthermore, our main theorem (\Cref{RegularGivesFFlat}, \Cref{FFlatGivesRegular}, and \Cref{AnimatedCase}) does not depend on a particular way of taking \(f_1, \dots, f_r\). This flexibility leads to good results, especially in the LCI case (\Cref{PrismaticKunzPrismLCI}), independent of the choice of \(f_1, \dots, f_r\).
\end{warning}

\begin{remark} \label{RemarkPrismAssumption}
    Building upon the observation in the preceding remark (\Cref{RemarkBoundedAssumption}), considering the case where \((A, I)\) is a bounded prism and \(R \defeq A/J\) where \(J = (I, f_1, \dots, f_r)\) plays a crucial role in various contexts.

    The primary instance in arithmetic geometry arises when \((A, I)\) is a perfect prism, signifying that \(R\) becomes a semiperfectoid ring.
    By assuming specific conditions on \(J\), such as \(f_1, \dots, f_r\) forming a regular sequence on \(A/I\), we introduce the concept of a (quasi)regular semiperfectoid ring \(R\).
    This concept has been extensively explored in recent studies in arithmetic geometry and algebraic K-theory (for example, \cite{bhatt2019Topologicala, anschutz2023Prismatic}).

    The next case is when \(A\) is a Noetherian ring (more precisely, a complete regular local ring).
    This case aligns with the domain of commutative algebra, notably due to Cohen's structure theorem (refer to \Cref{Envelope} and \Cref{Envelope2} for comprehensive details). 
    Furthermore, Shaul \cite{shaul2021Koszul} uses derived quotients to study the structure of a ``quotient'' by a sequence of elements possibly even a non-regular sequence.

\end{remark}

Next, we recall the notion of \emph{prismatic complex}.

\begin{definition}[Prismatic complex; cf. {\cite[\S 7.2]{bhatt2022Prismsa} and \cite[\S 4.1]{bhatt2022Absolute}}] \label{PrismaticComplexHodgeTate}
    Let \((A, I)\) be a bounded prism and let \(S\) be an animated \(A/I\)-algebra.
    The animated \(A/I\)-algebra \(S\) gives a \((p, I)\)-completed \(E_\infty\)-\(A\)-algebra \(\prism_{S/A}\), that is, \(\prism_{S/A}\) is a commutative algebra object in the \(\infty\)-category \(\Mod_A^{\wedge_{(p,I)}}\) of derived \((p, I)\)-complete \(A\)-modules (see \Cref{DefDerivedIComplete}).
    Furthermore, by \cite[Notation 7.10]{bhatt2022Prismatization} (or \cite[Example 2.4.15]{holeman2023Derived}), this \(E_\infty\)-\(A\)-algebra \(\prism_{S/A}\) can be equipped with the structure of a \((p, I)\)-complete derived \(\delta\)-\(A\)-algebra (the notion of \emph{derived (\(\delta\)-)rings} is described in \cite{raksit2020Hochschild, holeman2023Derived} following Mathew).
    We call the derived \(\delta\)-ring \(\prism_{S/A}\) as the \emph{prismatic complex of \(S\) relative to \(A\)}.

    If \(A \to \pi_0(S)\) is surjective, the cotangent complex \(L_{S/A}\) is \(1\)-connective and then the Hodge--Tate comparison below (\Cref{HodgeTateComparison}) shows that \(\overline{\prism}_{S/A}\) is connective and thus \(\prism_{S/A}\) is an animated \(\delta\)-ring.
    In this paper, we only consider this case, i.e., \(A \twoheadrightarrow \pi_0(S)\) is surjective and thus \(\prism_{S/A}\) is an animated \(\delta\)-\(A\)-algebra.

    Because of the \(\delta\)-structure, we can take a map of animated \(A\)-algebras \(\map{\varphi}{\prism_{S/A}}{\varphi_{A,*}\prism_{S/A}}\) which we call the \emph{Frobenius lift} on \(\prism_{S/A}\), where \(\varphi_{A,*}(-)\) is the restriction of scalars along the Frobenius lift \(\map{\varphi_A}{A}{\varphi_{A,*}A}\) on \(A\) (see also \Cref{PrismaticComplexAnimatedPrism} and \cite[Theorem 2.4.4]{holeman2023Derived}).
    We can also take an animated \(S\)-algebra \(\overline{\prism}_{S/A} \defeq \prism_{S/A} \otimes^L_A A/I\) which we call the \emph{Hodge--Tate complex of \(S\) relative to \(A\)}.
\end{definition}

\begin{definition}[Hodge--Tate comparison; {\cite[Remark 4.1.7]{bhatt2022Absolute}}] \label{HodgeTateComparison}
    The Hodge--Tate complex \(\overline{\prism}_{S/A}\) has the \emph{Hodge--Tate filtration} \(\{\Fil_i^{HT}\overline{\prism}_{S/A}\}_{i \geq 0}\) in the \(\infty\)-category \(\Mod_S^\wedge\) of derived \(p\)-complete \(S\)-modules.
    For each \(i \geq 0\) the cofiber \(\gr_i^{HT}\overline{\prism}_{S/A}\) of \(\Fil_{i-1}^{HT}\overline{\prism}_{S/A} \to \Fil_i^{HT}\overline{\prism}_{S/A}\) has the \emph{Hodge--Tate comparison isomorphism}
    \begin{equation*}
        \gr_i^{HT}\overline{\prism}_{S/A} \cong \parenlr{\parenlr{\bigwedge^i L_{S/(A/I)}}\{-i\}[-i]}^{\wedge}
    \end{equation*}
    in the \(\infty\)-category \(\Mod_{S}^\wedge\) of derived \(p\)-complete \(S\)-modules,
    where \(\wedge^i L_{S/(A/I)}\) is the \(i\)-th derived exterior power of the cotangent complex of \(S\) over \(A/I\) defined in \Cref{ExteriorPowerCotangentComplex}, \((-)\{-i\}\) is the twist \((-) \otimes_{A/I} (I/I^2)^{\otimes -i}\),
    and \((-)^{\wedge}\) is the derived \(p\)-completion functor on \(\Mod_S\) (see \Cref{DefDerivedCompletion} and \Cref{DerivedCompletionCompatible}).
\end{definition}

The next lemma is partially stated in Bhatt's lecture note in \cite[Definition VII.4.1]{bhatt2018Geometric}.
This relies on the formalism of ``\emph{animated prisms}'' and those theories introduced in \cite{bhatt2022Prismatization}.

\begin{lemma} \label{PrismaticComplexAnimatedPrism}
    Let \((A, I)\) be a bounded prism and let \(S\) be an animated \(A/I\)-algebra such that \(A/I \twoheadrightarrow \pi_0(S)\) is surjective.
    Then the prismatic complex \(\prism_{S/A}\) gives an animated prism \((\prism_{S/A} \to \overline{\prism}_{S/A})\) over \((A \to A/I)\).
    Together with the canonical map of animated \(A/I\)-algebras \(S \to \overline{\prism}_{S/A}\), the animated prism \((\prism_{S/A} \to \overline{\prism}_{S/A})\) is an object of the animated prismatic site \((S/A)^{an}_{\prism}\) defined in \cite[Construction 7.11]{bhatt2022Prismatization}.
\end{lemma}

\begin{proof}
    The prismatic complex \(\prism_{S/A}\) is already an animated \(\delta\)-\(A\)-algebra equipped with the Frobenius lift \(\varphi\) induced from the \(\delta\)-structure on the animated \(\delta\)-\(A\)-algebra \(\prism_{S/A}\) in \Cref{PrismaticComplexHodgeTate}.

    By taking base change, the animated \(\delta\)-\(A\)-algebra \(\prism_{S/A}\) gives an animated prism \((\prism_{S/A} \to \prism_{S/A}/I_{\prism_{S/A}})\) over \((A \to A/I)\) by \cite[Corollary 2.10]{bhatt2022Prismatization}, where \(I_{\prism_{S/A}} \defeq I \otimes^L_A \prism_{S/A}\).
    This animated prism is nothing but \((\prism_{S/A} \to \overline{\prism}_{S/A})\).
\end{proof}

The next theorem is one of our main theorems and a key to proving prismatic Kunz's theorem.
In particular, this is a generalization of \cite[Proposition 7.10]{bhatt2022Prismsa}, which needs to assume that \(R\) is a quasiregular semiperfectoid.

\begin{theorem} \label{AnimatedLCIPrism}
    Let \((A, I)\) be a bounded prism and let \(R \defeq A/J\), where \(J = (I, f_1, \dots, f_r)\) with a sequence of elements \(f_1, \dots, f_r\) in \(A\).
    Set an animated ring \(R^{an}(\underline{f}) = R^{an}(f_1, \dots, f_r)\) as in \Cref{ConstructionAnimatedQuotient}.
    Then \(\overline{\prism}_{R^{an}(\underline{f})/A}\) is isomorphic to the derived \(p\)-completion of a (possibly infinite) free \(R^{an}(\underline{f})\)-module and thus the canonical map of animated rings \(R^{an}(\underline{f}) \to \overline{\prism}_{R^{an}(\underline{f})/A}\) is \(p\)-completely faithfully flat in the sense of \Cref{PropOfAnimatedModule}.
    Furthermore, the canonical map of rings \(R/p^nR \to \pi_0(\overline{\prism}_{R^{an}(\underline{f})/A})/(p^n)\) is faithfully flat for all \(n \geq 1\).\footnote{This property is called \emph{adically faithfully flat} in the context of rigid geometry and satisfies effective descent condition for adically quasi-coherent sheaves (see \cite[Definition 4.8.12 (2) and Proposition 6.1.11]{fujiwara2018Foundations}).}
\end{theorem}

\begin{proof}
    If \(r = 0\), \(R = A/J\) is nothing but \(A/I\).
    In this case, \(\prism_{R^{an}(\underline{f})/A} = \prism_{(A/I)/A} \cong A\) by \cite[Example V.2.11]{bhatt2018Geometric} and thus \(\overline{\prism}_{R^{an}(\underline{f})/A} \cong A/I = R\) is a free \(R\)-module.
    Assume that \(r > 0\).
    By \Cref{ConstructionAnimatedQuotient}, \(R^{an}(\underline{f})\) is isomorphic to \(A/I \otimes^L_{\setZ[\underline{X}]} \setZ\) where \(A/I \leftarrow \setZ[\underline{X}] \to \setZ\) is defined by \(\overline{f_i} \mapsfrom X_i \mapsto 0\).
    The (algebraic) cotangent complex \(L_{R^{an}(\underline{f})/(A/I)}\)\footnote{The notion of an algebraic cotangent complex is equivalent to the usual cotangent complex in our case (see \Cref{EquivCotangentComplex}).} induces isomorphisms:
    \begin{align}
        L_{R^{an}(\underline{f})/(A/I)} & \cong L_{\setZ/\setZ[\underline{X}]} \otimes^L_\setZ R^{an}(\underline{f}) \label{CotangentComplexIsom} \\
        & \cong (X_1, \dots, X_r)/(X_1, \dots, X_r)^2[1] \otimes^L_\setZ R^{an}(\underline{f}) \cong (R^{an}(\underline{f}))^{\oplus r}[1] \nonumber
    \end{align}
    in the \(\infty\)-category \(\Mod_{R^{an}(\underline{f})}^{cn}\) of animated \(R^{an}(\underline{f})\)-modules by \cite[(5.1.8.1)]{cesnavicius2024Purity} (see also \cite[Proposition 2.3.8]{khan2019Virtual} under the perspective of derived algebraic geometry).

    For simplicity, we write the Hodge--Tate filtration \(\{\Fil_i^{HT}\overline{\prism}_{R^{an}(\underline{f})/A}\}_{i \geq 0}\) and those graded pieces \(\{\gr_i^{HT}\overline{\prism}_{R^{an}(\underline{f})/A}\}_{i \geq 0}\) as \(\{\Fil_i^{HT}\}_{i \geq 0}\) and \(\{\gr_i^{HT}\}_{i \geq 0}\).
    As mentioned in \Cref{HodgeTateComparison}, the Hodge--Tate filtration \(\{\Fil_i^{HT}\}_{i \geq 0}\) of \(\overline{\prism}_{R^{an}(\underline{f})/A}\) gives a Hodge--Tate comparison isomorphism
    \begin{equation} \label{IsomHodgeTate}
        \gr^{HT}_i \cong \parenlr{\parenlr{\bigwedge^i_{R^{an}(\underline{f})} L_{R^{an}(\underline{f})/(A/I)}}\{-i\}[-i]}^{\wedge} \in \Mod_{R^{an}(\underline{f})}.
    \end{equation}
    By using \Cref{ExteriorPowerCotangentComplex} and \cite[Proposition 25.2.4.2 and Corollary 25.2.3.2]{lurie2018Spectral}, the above isomorphisms (\ref{CotangentComplexIsom}) and (\ref{IsomHodgeTate}) imply that the isomorphisms of \(R^{an}(\underline{f})\)-modules for all \(i \geq 0\):
    \begin{align}
        \gr^{HT}_i & \cong \parenlr{\parenlr{\bigwedge^i_{R^{an}(\underline{f})} (R^{an}(\underline{f}))^{\oplus r}[1]}\{-i\}[-i]}^{\wedge} \label{IsomHodgeTateCotangent} \cong \parenlr{\parenlr{\Gamma^i_{R^{an}(\underline{f})}((R^{an}(\underline{f}))^{\oplus r})}\{-i\}}^{\wedge} \\
        & \cong \parenlr{(R^{an}(\underline{f}))^{\oplus D_i}\{-i\}}^{\wedge} \cong (R^{an}(\underline{f}))^{\oplus D_i}\{-i\} \label{IsomDividedPower}
    \end{align}
    where \(D_i \defeq \binom{r+i-1}{i}\) and \(\Gamma_{R^{an}(\underline{f})}^i(-)\) is the derived \(i\)-th divided power over \(R^{an}(\underline{f})\) defined in \cite[Construction 25.2.2.3]{lurie2018Spectral}.
    The last isomorphism follows from that the underlying complex of \(R^{an}(\underline{f})\) is actually a (homological) Koszul complex \(\Kos(A/I, f_1, \dots, f_r)\) which is already derived \(p\)-complete (see \Cref{DefDerivedIComplete}).

    For each \(j \geq 1\), we have a fiber sequence \(\Fil_{j-1}^{HT} \xrightarrow{\iota_j} \Fil_j^{HT} \xrightarrow{q_j} \gr_j^{HT}\) in \(\Mod_{R^{an}(\underline{f})}\).
    Since \(\gr_j^{HT}\) is connective by (\ref{IsomHodgeTateCotangent}) and \(\Fil^{HT}_0 = R^{an}(\underline{f})\), the fiber sequence shows that \(\Fil_j^{HT}\) is also connective for all \(j \geq 0\) and thus that is a fiber sequence in the \(\infty\)-category \(\Mod^{cn}_{R^{an}(\underline{f})}\) of animated \(R^{an}(\underline{f})\)-modules.
    Since \(\gr_j^{HT} \cong (R^{an}(\underline{f}))^{\oplus D_j}\) is a free animated \(R^{an}(\underline{f})\)-module (and thus a projective animated \(R^{an}(\underline{f})\)-module), there exists a right inverse \(\map{r_j}{\gr_j^{HT}}{\Fil_j^{HT}}\) of \(q_j\) (up to homotopy) in \(\Mod_{R^{an}(\underline{f})}^{cn}\) by \Cref{PropertiesModule}.
    As in \citeSta{05QT}, the fiber sequence induces an isomorphism \(\Fil_j^{HT} \cong \Fil_{j-1}^{HT} \oplus \gr_j^{HT}\) in \(\Mod_{R^{an}(\underline{f})}\) for any \(j \geq 1\).
    By \(\Fil^{HT}_0 = \gr^{HT}_0 \cong R^{an}(\underline{f})\), each \(\Fil^{HT}_j\) is isomorphic to \(\oplus_{0 \leq i \leq j} (R^{an}(\underline{f}))^{\oplus D_i}\{-i\}\) in \(\Mod_{R^{an}(\underline{f})}\).
    Since \(\overline{\prism}_{R^{an}(\underline{f})/A}\) is the (\(p\)-completed) colimit of \(\Fil_j^{HT}\) in the \(\infty\)-category \(\Mod_{R^{an}(\underline{f})}^\wedge\) of derived \(p\)-complete \(R^{an}(\underline{f})\)-modules, \(\overline{\prism}_{R^{an}(\underline{f})/A}\) is isomorphic to the derived \(p\)-completion \(\widehat{\oplus R^{an}(\underline{f})}\) of a free animated \(R^{an}(\underline{f})\)-module when considered as an animated \(R^{an}(\underline{f})\)-module.

    In particular, \(\overline{\prism}_{R^{an}(\underline{f})/A}\) is a \(p\)-completely faithfully flat animated \(R^{an}(\underline{f})\)-algebra: By the definition of \(p\)-completely faithfully flat modules over an animated ring (\Cref{PropOfAnimatedModule}), it suffices to show that \(\overline{\prism}_{R^{an}(\underline{f})/A} \otimes^L_{R^{an}(\underline{f})} \pi_0(R^{an}(\underline{f})) \in \Mod_{\pi_0(R^{an}(\underline{f}))} = \mcalD(R)\) is \(p\)-completely faithfully flat over \(R\) in the sense of \Cref{DefofCompFlat}.
    This follows from the following isomorphisms in \(D(R/pR)\) (see \cite[Lemma 4.4]{bhatt2019Topologicala}):
    \begin{align}
        \begin{split}
            (\overline{\prism}_{R^{an}(\underline{f})/A} \otimes^L_{R^{an}(\underline{f})} R) \otimes^L_R R/pR & \cong \overline{\prism}_{R^{an}(\underline{f})/A} \otimes^L_{R^{an}(\underline{f})} R/pR \\
            & \cong \overline{\prism}_{R^{an}(\underline{f})/A} \otimes^L_{R^{an}(\underline{f})} (R^{an}(\underline{f}) \otimes^L_\setZ \setZ/p\setZ) \otimes^L_{R^{an}(\underline{f}) \otimes^L_\setZ \setZ/p\setZ} R/pR\\
            & \cong (\widehat{\bigoplus R^{an}(\underline{f})} \otimes^L_\setZ \setZ/p\setZ) \otimes^L_{R^{an}(\underline{f}) \otimes^L_\setZ \setZ/p\setZ} R/pR \\
            & \cong (\bigoplus R^{an}(\underline{f}) \otimes^L_\setZ \setZ/p\setZ) \otimes^L_{R^{an}(\underline{f}) \otimes^L_\setZ \setZ/p\setZ} R/pR \cong \bigoplus R/pR.
        \end{split}
        \label{pnQuotient}
    \end{align}
    So the desired \(p\)-completely faithful flatness of \(R^{an}(\underline{f}) \to \overline{\prism}_{R^{an}(\underline{f})/A}\) holds.

    Since \(p\)-completely faithful flat maps are \(p^n\)-completely faithful flat for all \(n \geq 1\), \(\overline{\prism}_{R^{an}(\underline{f})/A} \otimes^L_{R^{an}(\underline{f})} R\) is a \(p^n\)-completely faithfully flat object in \(\mcalD(R)\) for all \(n \geq 1\).
    So we have faithfully flat maps of rings \(R/p^nR \to \pi_0(\overline{\prism}_{R^{an}(\underline{f})/A} \otimes^L_{R^{an}(\underline{f})} R/p^nR) \cong \pi_0(\overline{\prism}_{R^{an}(\underline{f})/A})/(p^n)\) for all \(n \geq 1\) as above.
    Here, the last isomorphism is deduced from the connectivity of \(R^{an}(\underline{f})\) and the Hodge--Tate complex \(\overline{\prism}_{R^{an}(\underline{f})/A}\), and \cite[Corollary 7.2.1.23 (2)]{lurie2017Higher}.
\end{proof}

In the above theorem, we can only show that \(R \to \pi_0(\overline{\prism}_{R^{an}(\underline{f})/A})\) is faithfully flat after modulo \(p^n\).
However, based on our conversation with Dine, the map itself is (\(p\)-completely) faithfully flat under a mild condition on \(R\) as follows.

\begin{corollary} \label{NoetherianCase}
    Let \((A, I)\) be a bounded prism and let \(R \defeq A/J\), where \(J = (I, f_1, \dots, f_r)\) with a sequence of elements \(f_1, \dots, f_r\) in \(A\).
    Set an animated ring \(R^{an}(\underline{f}) = R^{an}(f_1, \dots, f_r)\) as in \Cref{ConstructionAnimatedQuotient}.
    If \(R\) has bounded \(p^\infty\)-torsion (see \Cref{BoundedTorsion}), then the canonical map of rings \(R \to \pi_0(\overline{\prism}_{R^{an}(\underline{f})/A})\) is \(p\)-completely faithfully flat.
    Furthermore, if \(R\) is Noetherian, then \(R \to \pi_0(\overline{\prism}_{R^{an}(\underline{f})/A})\) is faithfully flat.
\end{corollary}

\begin{proof}
    Since \(R\) has bounded \(p^\infty\)-torsion and \(R = A/J\) is derived \(p\)-complete, \(R\) is \(p\)-adically complete by \cite[Lemma III.2.4]{bhatt2018Geometric}.
    By the above \Cref{AnimatedLCIPrism}, there exists a free \(R^{an}(\underline{f})\)-module \(M\) such that its derived \(p\)-completion \(\widehat{M}\) is isomorphic to \(\overline{\prism}_{R^{an}(\underline{f})/A}\).
    Since the free \(R\)-module \(\pi_0(M) \cong \oplus \pi_0(R^{an}(\underline{f})) \cong \oplus R\) also has bounded \(p^\infty\)-torsion, its derived \(p\)-completion \((\pi_0(M))^{\wedge}\) is discrete and coincides with the \(p\)-adic completion of the free \(R\)-module \(\pi_0(M)\).
    By \Cref{DerivedCompletionConnectedComponent}, \(\pi_0(\widehat{M}) \cong (\pi_0(M))^{\wedge}\) and thus \(\pi_0(\widehat{M}) \cong \pi_0(\overline{\prism}_{R^{an}(\underline{f})/A})\) is \(p\)-adically complete.
    We can apply \Cref{CompletionFaithfullyFlat} in this situation because of the faithful flatness of \(R/(p^n) \to \pi_0(\overline{\prism}_{R^{an}(\underline{f})/A})/(p^n)\) as above.
    Then \(\pi_0(\overline{\prism}_{R^{an}(\underline{f})/A})\) is \(p\)-completely faithfully flat over \(R\) (or faithfully flat if \(R\) is Noetherian).
\end{proof}

The bounded property of \(R\) is under consideration in \(p\)-adic Hodge theory \cite{bhatt2019Topologicala} and rigid geometry \cite{fujiwara2011Hausdorff}.

\begin{remark} \label{BoundedTorsion}
    Let \((A, I)\) be a bounded prism and let \(R \defeq A/J\), where \(J = (I, f_1, \dots, f_r)\) with a sequence of elements \(f_1, \dots, f_r\) in \(A\).
    If \(R\) is of characteristic \(0\) (this only means \(\setZ \subseteq R\)), we are interested in whether \(R\) has bounded \(p^\infty\)-torsion or not.
    The following cases are known. The former is a perfectoid flavor and the latter is a Noetherian flavor.
    \begin{enumalphp}
        \item \(R\) is a perfectoid ring or, more generally, a quasisyntomic ring.
        \item \(R\) is a Noetherian ring or, more generally, a rigid-Noetherian ring\footnote{A \(p\)-adic ring \(R\) is \emph{rigid-Noetherian} if \(R\) is \(p\)-adically complete and \(R[1/p]\) is Noetherian (\cite[Definition 5.1.1]{fujiwara2011Hausdorff}).}.
    \end{enumalphp}
    Of course, if \(R\) is Noetherian or quasisyntomic, then \(R\) has bounded \(p^\infty\)-torsion (\cite[Proposition 4.19 and Definition 4.20]{bhatt2019Topologicala}) (and thus \(R\) is \(p\)-adically complete since \(R\) is already derived \(p\)-complete).
    If \(R\) is rigid-Noetherian, then \(R\) is \(p\)-adically pseudo-adhesive\footnote{A ring \(R\) is \emph{\(p\)-adically pseudo-adhesive} if \(R[1/p]\) is Noetherian and any finitely generated \(R\)-module has bounded \(p^\infty\)-torsion (\cite[Definition 4.3.1]{fujiwara2011Hausdorff}).} by \cite[Theorem 5.1.2]{fujiwara2011Hausdorff}.
    So the \(R\)-module \(R\) itself has bounded \(p^\infty\)-torsion.
\end{remark}

\section{Prismatic Kunz's Theorem} \label{SectionPrismaticKunz}

Next, we apply the above statements to commutative algebra.

\subsection{The proof of prismatic Kunz's theorem}

We formulate ``\emph{prismatic Kunz's theorem}'', which characterizes the regularity of complete Noetherian local rings via the Frobenius lift of a prismatic complex (\Cref{AnimatedCase}).
We start providing some lemmas.

\begin{lemma}\label{617SatN}
    Let \((A, \mfrakm_A)\) be a complete Noetherian local domain.
    Assume that \(A\) has a $\delta$-structure such that \(\mfrakm_A\) is generated by elements \(p, x_1, \dots, x_n\) satisfying $\delta(x_{i})\in \mfrakm_A$.
    Then for an ideal $I\subseteq A$, the following conditions are equivalent. 
\begin{enumerate}
    \item $(A, I)$ is an orientable prism.
    \item $I$ is generated by a distinguished element $d\in \mfrakm_A$.
    \item $I$ is generated by $p-f$ for some $f\in (x_{1},\ldots, x_n)$.
\end{enumerate}
\end{lemma}

\begin{proof}
    \IFF{(1)}{(2)}: Since \(A\) is an integral domain, the principal ideal \(I\) is already free of rank \(1\).
    Thus this is a direct consequence of \cite[Lemma 2.25]{bhatt2022Prismsa}.

    \IMPLIES{(2)}{(3)}: By assumption, $p-d\in \mfrakm_A = (p, x_1, \dots, x_n)$. 
    We can write $p-d = pg+h$ by using some $g \in A$ and $h \in (x_{1},\ldots, x_{n})$. 
    Then $d=p(1-g)-h$ and $\delta(h)\in \mfrakm_A$ because of \(\delta(x_i) \in \mfrakm_A\).
    If $u:=1-g$ is a unit, we obtain a desired generator $u^{-1}d = p-u^{-1}h$ of $I$. 
    Assume that $u$ is not invertible.
    Then $pu \in p\mfrakm_A \subseteq \mfrakm_A^2$, which implies that $\delta(pu) \in \mfrakm_A$ by \Cref{NoDeltaStructure}. 
    Hence $\delta(d) = \delta(pu-h) \in \mfrakm_A$, but this is a contradiction. 

    \IMPLIES{(3)}{(2)}: Since $x_{i}$ and $\delta(x_{i})$ belong to $\mfrakm_A$ for any \(1 \leq i \leq n\), so is $\delta(f)$.
    Moreover, $\delta(p)$ is a unit and $(p, f)\subseteq \mfrakm_A$.
    A direct calculation of \(\delta(p-f)\) shows that $p-f$ is a distinguished element of \(A\), as desired. 
\end{proof}

\begin{example}[Complete log-regular rings] \label{LogRegularPrism}
    Following \cite{ishiro2025Perfectoida}, we can consider the next example.
    Let $R$ be a complete local log-regular ring of residue characteristic $p$ and let \(C\) be the Cohen ring of the residue field of \(R\) equipped with a \(\delta\)-structure (\Cref{CohenRing}).
    Then by Kato's structure theorem,
    \begin{equation*}
        R\cong C[|\mathcal{Q}|]/(p-f)
    \end{equation*}
    for some fine sharp saturated monoid $\mathcal{Q}$ and some $f\in C[|\mathcal{Q}|]$ with no non-zero constant terms.
    One can extend the Frobenius lift on \(C\) to the Frobenius lift on $C[|\mathcal{Q}|]$ by the rule: $e^{q}\mapsto (e^{q})^p$ ($q\in \mathcal{Q}$).
    Thus we obtain a $\delta$-structure on $C[|\mathcal{Q}|]$ such that $\delta(e^{q})=0$. 
    By Lemma \ref{617SatN}, $(C[|\mathcal{Q}|], (p-f))$ is an orientable prism (note that $C[|\mathcal{Q}|]$ is a domain). Moreover, since $R$ is a domain, that prism is transversal when $f\neq 0$, or crystalline otherwise. 
\end{example}

To apply the lemmas in \Cref{SectionPrismaticComplexes} for a complete Noetherian local ring, we use the following construction based on \Cref{617SatN}.

\begin{corollary}[{cf. \cite[Remark 3.11]{bhatt2022Prismsa}}] \label{Envelope}
    Let $(R, \mfrakm_R, k)$ be a complete Noetherian local ring of residue characteristic \(p\).
    Then there exists a surjective prismatic \(\delta\)-pair \((A, R)\) with a complete regular local prism \((A, I)\).
    Namely, there exists a map of $\delta$-pairs: 
    \begin{equation}\label{eq618SunN}
        (A, I) \to (A, J)
    \end{equation}
    where $A$ is a complete unramified regular local ring of mixed characteristic \((0, p)\), \(I\) and \(J\) are ideals of \(A\) such that \(I \subseteq J\), $(A, I)$ is a complete regular local prism, and $A/J \cong R$.
\end{corollary}

\begin{proof}
    We have a surjective map of rings $\alpha \colon C[|T_1,\ldots, T_n|]\twoheadrightarrow R$ that sends $\{T_{1}, \ldots, T_{n}\}$ to a system of generators of $\mfrakm_R$ where \(C\) is the Cohen ring of \(k\).
    Since $p\in \mfrakm_R$ and \(\mfrakm_R\) is generated by \(\alpha(T_1), \dots, \alpha(T_n) \in R\), there exists an element $f\in (T_{1},\ldots, T_{n}) \subseteq C[|T_1, \dots, T_n|]$ such that $p-f\in \ker (\alpha)$. 
    Here, $(C[|T_1,\ldots, T_n|], (p-f))$ is an orientable prism equipped with \(\delta(T_i) = 0\) by Lemma \ref{617SatN}. 
    Hence by putting $A \defeq C[|T_1,\ldots, T_n|]$, $I \defeq (p-f)$, and $J \defeq \ker (\alpha)$, we can take the desired map of $\delta$-pairs (\ref{eq618SunN}). 
\end{proof}

\begin{definition} \label{Envelope2}
    Let \((R, \mfrakm, k)\) be a complete Noetherian local ring of residue characteristic \(p\).
    Fix a surjective prismatic \(\delta\)-pair \((A, R)\) with a bounded prism \((A, I)\) such that \(\ker(A \twoheadrightarrow R) = (I, f_1, \dots, f_r)\) for some \(f_1, \dots, f_r \in A\), which always exists by \Cref{Envelope}. 

    We can take such a bounded prism \((A, I)\) satisfying that \((A, I)\) is transversal or crystalline, \(A/I\) is a Noetherian domain, and \(\dim(A) \leq \emdim(R) + 1\), where \(\emdim(R)\defeq \dim_k \mfrakm_R/\mfrakm_R^2\) is the embedding dimension of \(R\).
    We call it a \emph{small base prism (with respect to \(R\))}, which exists by the proof of \Cref{Envelope}.

    Furthermore, we can take a small base prism \((A, I)\) such that \((A, I)\) is a complete regular local prism and satisfies \(\dim(A) = \emdim(R) + 1\).
    We call such a small base prism a \emph{minimal complete regular local prism \((A, I)\) (with respect to \(R\))}.
\end{definition}

\begin{remark} \label{RemarkEnvelope}
    In certain cases, there are some choices of such a surjective prismatic \(\delta\)-pair.
    Based on \Cref{LogRegularPrism}, if \(R\) is a complete local log-regular ring of residue characteristic \(p\), we can take a surjective prismatic \(\delta\)-pair \((C[|\mcalQ|], R)\) with a small base prism \((C[|\mcalQ|], (p-f))\) which is not necessarily minimal but a small base prism.
\end{remark}

In the following, we use the next lemma which is a special case of \cite[Lemma 5.7]{bhatt2021CohenMacaulayness}.
We give a proof for the reader's convenience in our case.
See also \Cref{DerivedCompletionConnectedComponent} as a similar statement.

\begin{lemma} \label{DerivedCompletionDiscrete}
    Let \(M\) be an object in \(D^{\leq 0}(\setZ)\).
    If the derived \(p\)-completion \(\widehat{M}\) is discrete and \(p\)-torsion-free, then the isomorphism
    \begin{equation*}
        \widehat{M} \cong \widehat{H^0(M)}
    \end{equation*}
    holds, where the right-hand side is the \(p\)-adic completion of \(H^0(M)\).
\end{lemma}

\begin{proof}
    Taking the exact triangle \(\tau^{\leq -1}M \to M \to \tau^{\geq 0}M \xrightarrow{+1}\) in \(D(\setZ)\).
    By our assumption, we have the following isomorphism for each \(n \geq 1\):
    \begin{align}
        M \otimes^L_\setZ \setZ/p^n\setZ & \cong \widehat{M} \otimes^L_\setZ \setZ/p^n\setZ \cong \widehat{M}/p^n\widehat{M}[0]. \label{CenterIsom}
    \end{align}
    Since \((\tau^{\leq -1}M) \otimes^L_\setZ \setZ/p^n\setZ\) is concentrated in degree \(\leq -1\), the cohomological long exact sequence of \((\tau^{\leq -1}M) \otimes^L_\setZ \setZ/p^n\setZ \to M \otimes^L_\setZ \setZ/p^n\setZ \to (\tau^{\geq 0}M) \otimes^L_\setZ \setZ/p^n\setZ \xrightarrow{+1}\) shows that the map \(\widehat{M}/p^n\widehat{M} \to H^0(M)/p^nH^0(M)\) is an isomorphism.
    Combining this and \eqref{CenterIsom}, we have the following isomorphism:
    \begin{align*}
        \widehat{M} & = H^0(\widehat{M}) \cong H^0(R\lim_n(M \otimes^L_\setZ \setZ/p^n\setZ)) \cong H^0(R\lim_n(\widehat{M} \otimes^L_\setZ \setZ/p^n\setZ)) \\
        & \cong \lim_n(H^0(M)/p^nH^0(M)) = \widehat{H^0(M)}.
    \end{align*}
    This shows the desired isomorphism.
\end{proof}

Now we can prove the following variant of ``\emph{prismatic Kunz's theorem}''.
We first show that any regular local ring gives a faithfully flat Frobenius lift of a prismatic complex.
The main technique is the deformation property of regular prisms (\Cref{EquivRegularPrism}) and (classical) Kunz's theorem (\Cref{KunzFrob}).

\begin{theorem}[Regular local ring gives the faithfully flat Frobenius lift] \label{RegularGivesFFlat}
    Let \((R, \mfrakm)\) be a complete Noetherian local ring of residue characteristic \(p\).
    Fix a surjective prismatic \(\delta\)-pair \((A, R)\) with a small base prism \((A, I)\) in the sense of \Cref{Envelope2}.
    If \(R\) is a regular local ring, then the Frobenius lift \(\map{\varphi}{\prism_{R/A}}{\varphi_{A,*}\prism_{R/A}}\) of the animated \(\delta\)-\(A\)-algebra \(\prism_{R/A}\) is faithfully flat.
    In this case, \(A/I \twoheadrightarrow R\) is an isomorphism and \(\prism_{R/A}\) is \(A\) itself.
\end{theorem}

\begin{proof}
    Since \(R\) is a regular local ring, the small base prism \((A, I)\) satisfies
    \begin{equation*}
        \dim(R) \leq \dim(A/I) = \dim(A) - 1 \leq \emdim(R) = \dim(R).
    \end{equation*}
    Then the surjective map \(A/I \twoheadrightarrow R\) is an isomorphism because \(A/I\) is an integral domain by the assumption of small base prisms (see \Cref{Envelope2}).
    In particular, \(\prism_{R/A} \cong A\) by \cite[Example V.2.11]{bhatt2018Geometric}.
    So the Frobenius lift \(\varphi\) on \(\prism_{R/A} = \pi_0(\prism_{R/A}) = A\) is faithfully flat by the following reason: Since \(R \cong A/I\) is a complete regular local ring, so is \(A/pA\) by \Cref{EquivRegularPrism}.
    By Kunz's theorem, the Frobenius map \(F\) on \(A/pA\) is faithfully flat and thus \(\varphi\) is \(p\)-completely faithfully flat by \Cref{TransversalFFlatEquiv}.
    This implies that \(\varphi\) is faithfully flat by \cite[Proposition 5.1]{bhatt2018Direct} since \(A\) is a \(p\)-torsion-free Noetherian ring.
\end{proof}

Next, we show that the faithfully flat Frobenius lift of a prismatic complex gives a regular local ring.
The main technique is our theorem above (\Cref{AnimatedLCIPrism}) and (a special case of) \(p\)-adic Kunz's theorem as follows (\Cref{KunzRegular}) which is a generalization of Kunz's theorem (\Cref{KunzFrob}) in terms of perfect closure.

\begin{theorem}[\(p\)-adic Kunz's theorem; {\cite[Theorem 4.7]{bhatt2019Regular}}] \label{KunzRegular}
    Let \(R\) be a \(p\)-Zariskian Noetherian ring.
    Then \(R\) is regular if and only if \(R\) has a faithfully flat map \(R \to A\) to a perfectoid ring \(A\).
\end{theorem}

The following implication can be proved without assuming that \((A, I)\) is a small base prism or even more transversal or crystalline.

\begin{theorem}[Faithful flatness gives the regularity of rings] \label{FFlatGivesRegular}
    Let \((R, \mfrakm)\) be a complete Noetherian local ring of residue characteristic \(p\).
    Fix a surjective prismatic \(\delta\)-pair \((A, R)\) with a bounded prism \((A, I)\) such that \(\ker(A \twoheadrightarrow R) = (I, f_1, \dots, f_r)\) for some \(f_1, \dots, f_r \in A\) (which need not be transversal or crystalline).
    Set an animated ring \(R^{an}(\underline{f}) = R^{an}(f_1, \dots, f_r)\) as in \Cref{ConstructionAnimatedQuotient}.
    If the Frobenius lift \(\map{\pi_0(\varphi)}{\pi_0(\prism_{R^{an}(\underline{f})/A})}{\varphi_{A,*}\pi_0(\prism_{R^{an}(\underline{f})/A})}\) of the \(\delta\)-ring \(\pi_0(\prism_{R^{an}(\underline{f})/A})\) induced from the animated \(\delta\)-structure on \(\prism_{R^{an}(\underline{f})/A}\) is \(p\)-completely faithfully flat, then \(R\) is a regular local ring.
\end{theorem}

\begin{proof}
    By \Cref{PrismaticComplexAnimatedPrism}, \((\prism_{R^{an}(\underline{f})/A} \to \overline{\prism}_{R^{an}(\underline{f})/A})\) is an animated prism over \((A \to A/I)\).
    Taking the colimit in the \(\infty\)-category of derived \(p\)-complete animated \(\delta\)-rings, we have a (derived \(p\)-complete) perfect animated \(\delta\)-ring \((\colim_{\varphi} \prism_{R^{an}(\underline{f})/A})^{\wedge_p}\).
    This is actually a Witt ring \(W(P)\) for some perfect ring \(P\) by \cite[Remark A.17]{bhatt2022Prismatization}, in particular, is \(p\)-torsion-free.
    Consequently, we can apply \Cref{DerivedCompletionDiscrete} for the colimit \( M \defeq \colim_{\varphi} \prism_{R^{an}(\underline{f})/A}\) since \(M\) is connective and its derived \(p\)-completion is concentrated in degree \(0\) and \(p\)-torsion-free.
    Thus, we have an isomorphism of discrete \(A\)-algebras
    \begin{equation}
        (\colim_{\varphi} \prism_{R^{an}(\underline{f})/A})^{\wedge_p} \cong (\pi_0(\colim_{\varphi} \prism_{R^{an}(\underline{f})/A}))^{\wedge_p} \label{ColimitComplete}
    \end{equation}
    induced from the canonical map \(\colim_{\varphi} \prism_{R^{an}(\underline{f})/A} \to \pi_0(\colim_{\varphi} \prism_{R^{an}(\underline{f})/A})\), where the left-hand side is the derived \(p\)-completion and the right-hand side is the \(p\)-adic completion.
    Moreover, since \(\pi_0(\colim_{\varphi} \prism_{R^{an}(\underline{f})/A})\) is a perfect \(\delta\)-ring, it is \(p\)-torsion-free by \cite[Lemma 2.28]{bhatt2022Prismsa} and then the right-hand side of (\ref{ColimitComplete}) can be seen as the derived \(p\)-completion of \(\pi_0(\colim_{\varphi} \prism_{R^{an}(\underline{f})/A})\).

    Furthermore, taking the colimit \(\prism_{R^{an}(\underline{f})/A, \infty} \defeq (\colim_{\varphi} \prism_{R^{an}(\underline{f})/A})^{\wedge_{(p, I)}}\) in the \(\infty\)-category of derived \((p, I)\)-complete animated \(\delta\)-rings, \(\prism_{R^{an}(\underline{f})/A, \infty}\) becomes a perfect animated \(\delta\)-ring.
    As in \Cref{PrismaticComplexAnimatedPrism}, this gives a perfect animated prism \((\prism_{R^{an}(\underline{f})/A, \infty} \to \overline{\prism}_{R^{an}(\underline{f})/A, \infty} \defeq \prism_{R^{an}(\underline{f})/A, \infty}/I\prism_{R^{an}(\underline{f})/A, \infty})\) by the ``rigidity'' of prismatic structure (\cite[Corollary 2.10]{bhatt2022Prismatization}).
    Since perfect animated prisms are identified with perfect prisms by \cite[Corollary 2.18]{bhatt2022Prismatization}, this perfect animated prism \((\prism_{R^{an}(\underline{f})/A, \infty} \to \overline{\prism}_{R^{an}(\underline{f})/A, \infty})\) is actually a perfect (discrete) prism  \((\prism_{R^{an}(\underline{f})/A, \infty}, I\prism_{R^{an}(\underline{f})/A, \infty})\).
    In particular, \(\overline{\prism}_{R^{an}(\underline{f})/A, \infty}\) is a perfectoid ring.

    The \(p\)-completely faithful flatness of \(\varphi\) on \(\pi_0(\prism_{R^{an}(\underline{f})/A})\) implies that the canonical map
    \begin{equation}
        \pi_0(\prism_{R^{an}(\underline{f})/A}) \to (\colim_{\varphi} \pi_0(\prism_{R^{an}(\underline{f})/A}))^{\wedge_{p}} \cong (\colim_{\varphi} \prism_{R^{an}(\underline{f})/A})^{\wedge_{p}} \label{ColimitMap}
    \end{equation}
    is a \(p\)-completely faithfully flat map of rings, where the last isomorphism follows from (\ref{ColimitComplete}).
    Since \(\prism_{R^{an}(\underline{f})/A, \infty}\) is the derived \(I\)-completion of \((\colim_{\varphi} \prism_{R^{an}(\underline{f})/A})^{\wedge_p}\), these are isomorphic each other after taking the base change \((-) \otimes^L_A A/I\).
    Note that \(\pi_0(\overline{\prism}_{R^{an}(\underline{f})/A}) \cong \pi_0(\prism_{R^{an}(\underline{f})/A}) \otimes_A A/I\) because of the connectivity of \(\prism_{R^{an}(\underline{f})/A}\).
    By using these two isomorphisms, the base change \(\pi_0(\overline{\prism}_{R^{an}(\underline{f})/A}) \to (\colim_{\varphi} \prism_{R^{an}(\underline{f})/A})^{\wedge_{p}} \otimes^L_A A/I \cong \overline{\prism}_{R^{an}(\underline{f})/A, \infty}\) of (\ref{ColimitMap}) is also \(p\)-completely faithfully flat.

    Combining this and \Cref{NoetherianCase}, we have a \(p\)-completely faithfully flat map of rings
    \begin{equation} \label{RtoPrism}
        R \to \pi_0(\overline{\prism}_{R^{an}(\underline{f})/A}) \to \overline{\prism}_{R^{an}(\underline{f})/A, \infty}.
    \end{equation}
    Since \(R\) is Noetherian and the perfectoid ring \(\overline{\prism}_{R^{an}(\underline{f})/A, \infty}\) is \(p\)-adically complete, the map \(R \to \overline{\prism}_{R^{an}(\underline{f})/A, \infty}\) is faithfully flat by \Cref{CompletionFaithfullyFlat}.
    Then \(R\) is a regular local ring by \(p\)-adic Kunz's theorem (\Cref{KunzRegular}).
\end{proof}

Combining these theorems, we have the following equivalence which we call ``\emph{prismatic Kunz's theorem}''.

\begin{corollary}[Prismatic Kunz's theorem] \label{AnimatedCase}
    Let \((R, \mfrakm)\) be a complete\footnote{Recall that a Noetherian local ring \((R, \mfrakm)\) is a regular local ring if and only if the \(\mfrakm\)-adic completion \(\widehat{R}^{\mfrakm}\) is a regular local ring. So we can apply this theorem for any Noetherian local ring \((R, \mfrakm)\) after taking \(\mfrakm\)-adic completion.} Noetherian local ring of residue characteristic \(p\).
    Fix a surjective prismatic \(\delta\)-pair \((A, R)\) with a small base prism \((A, I)\) and fix a sequence of elements \(f_1, \dots, f_r\) in \(A\) such that \(\ker(A \twoheadrightarrow R) = (I, f_1, \dots, f_r)\).\footnote{Note that this corollary is independent of the choices of such a \(\delta\)-pair \((A, R)\) and a sequence of elements \(f_1, \dots, f_r\) because of the proof of the above theorems.}
    Set an animated ring \(R^{an}(\underline{f}) = R^{an}(f_1, \dots, f_r)\).
    Then the following are equivalent:
    \begin{enumerate}
        \item \(R\) is a (complete) regular local ring.
        \item The Frobenius lift \(\map{\varphi}{\prism_{R^{an}(\underline{f})/A}}{\varphi_{A,*}\prism_{R^{an}(\underline{f})/A}}\) of the animated \(\delta\)-\(A\)-algebra \(\prism_{R^{an}(\underline{f})/A}\) is faithfully flat.
        \item The Frobenius lift \(\map{\pi_0(\varphi)}{\pi_0(\prism_{R^{an}(\underline{f})/A})}{\varphi_{A,*}\pi_0(\prism_{R^{an}(\underline{f})/A})}\) of the \(\delta\)-ring \(\pi_0(\prism_{R^{an}(\underline{f})/A})\) induced from the animated \(\delta\)-structure on \(\prism_{R^{an}(\underline{f})/A}\) is (\(p\)-completely) faithfully flat.
    \end{enumerate}
\end{corollary}

\begin{proof}
    \IMPLIES{(1)}{(2)} is proved in \Cref{RegularGivesFFlat} and \IMPLIES{(2)}{(3)} is trivial by the definition of the faithful flatness of maps of animated rings (\Cref{PropOfAnimatedModule} (3)).
    \IMPLIES{(3)}{(1)} is proved in \Cref{FFlatGivesRegular}.
\end{proof}

If \(R\) is a complete intersection, we do not need to use the notion of animated rings and animated prisms as follows.

\begin{lemma} \label{DiscretePrismaticComplex}
    Let \(R\) be a complete Noetherian local ring of residue characteristic \(p\).
    Assume further that \(R\) is a complete intersection.
    Fix a surjective prismatic \(\delta\)-pair \((A, R)\) with a complete regular local prism \((A, I)\).
    Since \(R\) is a complete intersection, we can fix a representation \(\ker(A \twoheadrightarrow R) = (I, f_1, \dots, f_r)\) such that \(f_1, \dots, f_r\) is a regular sequence on \(A/I\) by \citeSta{09PZ}.
    Set an animated ring \(R^{an}(\underline{f}) = R^{an}(f_1, \dots, f_r)\).
    Then we have the following.
    \begin{enumalphp}
        \item The animated ring \(R^{an}(\underline{f})\) is actually a usual ring \(R\).
        \item The animated prism \((\prism_{R^{an}(\underline{f})/A} \to \overline{\prism}_{R^{an}(\underline{f})/A})\) is a bounded orientable discrete prism \((\prism_{R/A}, I\prism_{R/A})\).
        \item If \(R\) is \(p\)-torsion-free, then \((\prism_{R/A}, I\prism_{R/A})\) is transversal.
        \item If \(R\) is of characteristic \(p\), then \((\prism_{R/A}, I\prism_{R/A})\) is crystalline.
    \end{enumalphp}
\end{lemma}

\begin{proof}
    Since \(f_1, \dots, f_r\) is a regular sequence on \(A/I\), the animated ring \(R^{an}(\underline{f})\) does not have higher homotopy groups and thus is isomorphic to a usual ring \(R\). This shows (a).

    By \Cref{AnimatedLCIPrism}, the Hodge--Tate complex \(\overline{\prism}_{R^/A} \cong \widehat{\oplus R}\) is the \(p\)-adic completion of a free \(R\)-module and in particular discrete.
    By \cite[Lemma 7.7 (3)]{bhatt2022Prismsa}, \(\prism_{R/A}\) is also concentrated in degree \(0\) and makes an orientable prism \((\prism_{R/A}, \xi \prism_{R/A})\) where \(\xi\) is an orientation of \(I\) in \(A\). This shows (b).
    Since \(\xi \prism_{R/A}\) is locally free of rank \(1\) on \(\prism_{R/A}\), \(\xi\) is a non-zero-divisor of \(\prism_{R/A}\).
    If \(R\) is \(p\)-torsion-free, \(\overline{\prism}_{R/A} \cong \widehat{\oplus R}\) is also \(p\)-torsion-free.
    Thus \((\prism_{R/A}, I\prism_{R/A})\) is transversal by \Cref{DefofTransversal} and this shows (c).
    If \(R\) is of characteristic \(p\), \(\xi\) can be taken as \(p\) and this shows (d).
\end{proof}

Under the complete intersection assumption as in \Cref{DiscretePrismaticComplex}, the above theorem deduces the next corollary.

\begin{corollary}[lci case] \label{PrismaticKunzPrismLCI}
    Let \(R\) be a complete Noetherian local ring of residue characteristic \(p\).
    Assume further that \(R\) is a complete intersection.
    Fix a surjective prismatic \(\delta\)-pair \((A, R)\) with a complete regular local small base prism \((A, I)\).
    Then the following are equivalent:
    \begin{enumerate}
        \item \(R\) is a (complete) regular local ring.
        \item The Frobenius lift \(\map{\varphi}{\prism_{R/A}}{\varphi_{A,*}\prism_{R/A}}\) of the \(\delta\)-ring \(\prism_{R/A}\) (see \Cref{DiscretePrismaticComplex}) is (\(p\)-completely) faithfully flat.
    \end{enumerate}
    If we assume further that \(R\) is \(p\)-torsion-free or of characteristic \(p\), then the following conditions are also equivalent:
    \begin{enumerate}
        \item \(R\) is a regular local ring.
        \item \(\varphi\) is faithfully flat.
        \item \(\varphi\) is \(p\)-completely faithfully flat.
        \item \(\varphi\) is \(I\)-completely faithfully flat.
        \item \(\varphi\) is \((p, I)\)-completely faithfully flat.
        \item The Frobenius map \(\map{F}{\prism_{R/A}/(p)}{F_*(\prism_{R/A}/(p))}\) is faithfully flat.
        \item The map \(\map{\overline{\varphi}_I}{\prism_{R/A}/I}{\varphi_*(\prism_{R/A}/\varphi(I))}\) induced from \(\varphi\) is faithfully flat.
        \item The \(p\)-th power map \(\map{\overline{\varphi}_{(p,I)}}{\prism_{R/A}/(p, I)}{\varphi_*(\prism_{R/A}/(p, I^{[p]}))}\) induced from \(\varphi\) is faithfully flat.
    \end{enumerate}
\end{corollary}

\begin{proof}
    The first equivalence is a direct consequence of \Cref{AnimatedCase} and \Cref{DiscretePrismaticComplex}.

    If \(R\) is \(p\)-torsion-free or of characteristic \(p\), \((\prism_{R/A}, I\prism_{R/A})\) is a transversal or crystalline prism by \Cref{DiscretePrismaticComplex}.
    So the equivalence in this setting is a direct consequence of \Cref{CompletionFaithfullyFlat} and \Cref{AnimatedCase}.
\end{proof}

If we do not assume the property that \((A, I)\) is a small base (or minimal) prism, our theorem does not hold in general.
The following is a counterexample suggested by Bhargav Bhatt.

\begin{remark}[Bhatt] \label{CounterExample}
    If we do not assume the assumption that \((A, I)\) is a small base prism, the above \Cref{AnimatedCase} and \Cref{PrismaticKunzPrismLCI} may fail.
    Here is a counterexample.

    For \(R \defeq \setF_p\), the surjective map \(A \defeq \setZ_p[|T|] \twoheadrightarrow \setF_p; T \mapsto 0\) gives a surjective prismatic \(\delta\)-pair \((\setZ_p[|T|], \setF_p)\) with a complete regular local prism \((\setZ_p[|T|], (p))\), where the \(\delta\)-structure (equivalently, the Frobenius lift \(\varphi\)) on \(\setZ_p[|T|]\) is given by \(\restr{\varphi}{\setZ_p} = \id_{\setZ_p}\) and \(\varphi(T) = T^p\).
    However, \((\setZ_p[|T|], (p))\) is not a small base prism and not even minimal because of \(\dim(\setZ_p[|T|]) = 2 > \emdim(\setF_p) + 1 = 1\).
    In this case, \(R\) is a regular local ring but the Frobenius lift \(\varphi\) on \(\prism_{R/A}\) is not faithfully flat:
    since \(T \in \setZ_p[|T|] = A\) is \emph{\((p)\)-completely regular relative to \(A_0 \defeq \setZ_p\)}, that is, \(\setF_p \cong A_0/^L(p) \to A/^L(p, T) \cong \setF_p\) is a flat map of (animated) rings (see \cite[Definition 2.42]{bhatt2022Prismsa}), \(\prism_{R/A}\) is the prismatic envelope (or the derived \(p\)-completed \(\delta\)-\(A\)-algebra obtained by freely adjoining \(T/p\))
    \begin{equation*}
        \prism_{R/A} = A\bracelr{\frac{(p, T)}{(p)}}^\wedge \cong A\bracelr{\frac{T}{p}}^\wedge
    \end{equation*}
    by \cite[Proposition 3.13 and Example 7.9]{bhatt2022Prismsa}.
    Taking \(p\)-completed \(\varphi_A\)-pullback \(R^{(1)} \cong A/(p, T^p)\) gives isomorphisms of \(\delta\)-\(A\)-algebras
    \begin{equation*}
        \varphi_A^* \prism_{R/A} \cong \prism_{R^{(1)}/A} \cong \prism_{(A/(p, T^p))/A} \cong A\bracelr{\frac{T^p}{p}}^\wedge \cong A\bracelr{\frac{\varphi_A(T)}{p}} \cong D_{(T)}(A),
    \end{equation*}
    where \(D_{(T)}(A)\) is the pd-envelope\footnote{The basic knowledge of \emph{pd-envelopes} (or \emph{divided power envelopes}) as sufficient for our purposes is summarized in \cite[Lecture VI]{bhatt2018Geometric}. Explicitly, \(D_{(T)}(A)\) is the \(p\)-adic completion of the subring \(A[\gamma_n(T)]_{n \geq 1}\) of \(A[1/p]\) generated by \(A\) and \(\gamma_n(T) \defeq T^n/n!\).} of \((T) \subseteq A\) and the last isomorphism is deduced from \cite[Corollary 2.39]{bhatt2022Prismsa}.
    If the Frobenius lift \(\varphi\) of \(\prism_{R/A}\) is faithfully flat, the Frobenius lift of \(D_{(T)}(A)\) which is given by \(\varphi(\gamma_n(T)) = \gamma_n(T^p)\) is also faithfully flat by the above isomorphisms of \(\delta\)-\(A\)-algebras.
    However, the Frobenius lift of \(D_{(T)}(A)\) is not faithfully flat because \(D_{(T)}(A)/(p)\) has a non-zero nilpotent element \(T\).
    So the Frobenius lift \(\varphi\) of \(\prism_{R/A}\) is not faithfully flat.
\end{remark}

\subsection{Applications} \label{Application}
First, similar to Kunz's theorem, the regularity of prisms is characterized by the faithful flatness of the Frobenius lift.
We give two proofs of this corollary, one is a direct consequence of \Cref{AnimatedCase}, and the other is a direct consequence of (classical) Kunz's theorem.

\begin{proposition} \label{PrismaticKunzCor}
    Let \((A, I)\) be a prism such that \(A\) is a Noetherian local ring with the maximal ideal \(\mfrakm\).
    Let \(\map{\varphi_A}{A}{\varphi_{A,*} A}\) be the Frobenius lift of \(A\).
    Then the following are equivalent:
    \begin{enumerate}
        \item \((A, I)\) is a regular local prism. \label{PrismaticKunzCor-RegPrism}
        \item \(\varphi_A\) is faithfully flat. \label{PrismaticKunzCor-FFlat}
        \item \(\varphi_A\) is \(p\)-completely faithfully flat. \label{PrismaticKunzCor-pFFlat}
        \item \(\varphi_A\) is \(I\)-completely faithfully flat. \label{PrismaticKunzCor-IFFlat}
        \item The Frobenius map \(\map{F}{A/pA}{F_* (A/pA)}\) is faithfully flat. \label{PrismaticKunzCor-FrobFFlat}
        \item The prism \((A, I)\) is transversal or crystalline, and the map \(\map{\overline{\varphi}_{A,I}}{A/I}{\varphi_* (A/\varphi(I))}\) induced from \(\varphi_A\) is faithfully flat. \label{PrismaticKunzCor-ModIFFlat}
        \item The prism \((A, I)\) is transversal or crystalline, and the \(p\)-th power map \(\map{\overline{\varphi}_{A,(p,I)}}{A/(p, I)}{\varphi_{A,*}(A/(p, I^{[p]}))}\) induced from \(\varphi_A\) is faithfully flat. \label{PrismaticKunzCor-ModIpFFlat}
    \end{enumerate}
\end{proposition}

\begin{proof}
    The \(\mfrakm\)-adic completion \(\widehat{A}^{\mfrakm}\) of \(A\) also makes a prism \((\widehat{A}^{\mfrakm}, I\widehat{A}^{\mfrakm})\) because of \cite[Lemma 2.17]{bhatt2022Prismsa} and the faithfully flat map of rings \(A \to \widehat{A}^{\mfrakm}\).
    So \((A, I)\) is a regular local prism if and only if \((\widehat{A}^{\mfrakm}, I\widehat{A}^{\mfrakm})\) is a complete regular local prism.
    Similarly, the Frobenius lift \(\map{\varphi_A}{A}{\varphi_{A,*} A}\) is faithfully flat if and only if so is \(\map{\varphi_{\widehat{A}^{\mfrakm}}}{\widehat{A}^{\mfrakm}}{\varphi_{\widehat{A}^{\mfrakm},*} \widehat{A}^{\mfrakm}}\).
    Without loss of generality, we can assume that \(A\) is a complete Noetherian local ring.
    Note that \(\prism_{(A/I)/A} \cong A\).
    
    We first show \IFF{(\ref{PrismaticKunzCor-RegPrism})}{(\ref{PrismaticKunzCor-FFlat})}.
    If \((A, I)\) is a regular local prism, \(A/I\) is a (complete) regular local ring by \Cref{EquivRegularPrism}.
    In particular, \(A/I\) is an integral domain and thus we have a surjective prismatic \(\delta\)-pair \((A, A/I)\) with a small base prism \((A, I)\).
    Applying \Cref{AnimatedCase}, the Frobenius lift \(\varphi = \varphi_A\) of \(\prism_{(A/I)/A} \cong A\) is faithfully flat.

    Conversely, if \(\varphi_A\) is faithfully flat, the isomorphism \(\prism_{(A/I)/A} \cong A\) also shows that \(\map{\varphi}{\prism_{(A/I)/A}}{\varphi_{A,*}\prism_{(A/I)/A}}\) is faithfully flat.
    Again applying \Cref{AnimatedCase}, \(A/I\) is a regular local ring, and thus \((A, I)\) is a regular local prism.

    Other equivalences are straightforward.
    By \Cref{TransversalFFlatEquiv}, it is sufficient to show \IMPLIES{(\ref{PrismaticKunzCor-FrobFFlat})}{(\ref{PrismaticKunzCor-RegPrism})}, which is a direct consequence of \Cref{EquivRegularPrism}: If (\ref{PrismaticKunzCor-FrobFFlat}) holds, \(A/pA\) is a regular local ring by Kunz's theorem and thus \((A, I)\) is a regular local prism.

    (Another proof): While the above proof uses \Cref{AnimatedCase}, we can deduce this theorem by using Kunz's theorem (\Cref{KunzFrob}).
    Here is such a proof.
    We can assume that \((A, I)\) is a bounded orientable prism since \(A\) is a Noetherian local ring and \(I\) is an invertible ideal.
    We show the following implications:
    \begin{center}
        \begin{tikzcd}
            (\ref{PrismaticKunzCor-RegPrism}) \arrow[r, Leftrightarrow] \arrow[d, Leftrightarrow] & (\ref{PrismaticKunzCor-pFFlat}) \arrow[r, Leftrightarrow] & (\ref{PrismaticKunzCor-FFlat}) \arrow[d, Rightarrow]  & (\ref{PrismaticKunzCor-IFFlat}) \arrow[l, Leftrightarrow] \\
            (\ref{PrismaticKunzCor-FrobFFlat}) &  & (\ref{PrismaticKunzCor-ModIFFlat}) \arrow[r, Rightarrow] & (\ref{PrismaticKunzCor-ModIpFFlat}). \arrow[u, Rightarrow]
        \end{tikzcd}
    \end{center}

    The equivalence (\ref{PrismaticKunzCor-pFFlat}) \(\Leftrightarrow\) (\ref{PrismaticKunzCor-FFlat}) \(\Leftrightarrow\) (\ref{PrismaticKunzCor-IFFlat}) follows from the definition of completely flatness and \cite[Proposition 5.1]{bhatt2018Direct} since \(A\) is Noetherian and \(p\)-torsion-free: We use the fact that, if \(\varphi_A\) is \(p\)-completely faithfully flat, the derived tensor product \(\varphi_{A,*}A \otimes^L_A A/pA\) is concentrated in degree \(0\) and thus \(A\) is \(p\)-torsion-free.

    (\ref{PrismaticKunzCor-RegPrism}) \(\Leftrightarrow\) (\ref{PrismaticKunzCor-FrobFFlat}): This equivalence is deduced from Kunz's theorem and \Cref{EquivRegularPrism}.

    (\ref{PrismaticKunzCor-RegPrism}) \(\Rightarrow\) (\ref{PrismaticKunzCor-pFFlat}): Since \(A/pA\) is a regular local ring by \Cref{EquivRegularPrism}, the Frobenius map \(\map{F}{A/p A}{A/p A}\) is faithfully flat by Kunz's theorem.
    Note that any regular local prism is transversal or crystalline.
    In both cases, \Cref{TransversalFFlatEquiv} shows (\ref{PrismaticKunzCor-RegPrism}) \(\Rightarrow\) (\ref{PrismaticKunzCor-pFFlat}).

    (\ref{PrismaticKunzCor-pFFlat}) \(\Rightarrow\) (\ref{PrismaticKunzCor-RegPrism}): By the definition of \(p\)-completely flatness, the Frobenius map \(\map{F}{A/p A}{A/p A}\), which is induced from \(\varphi_A\), is flat.
    Since \(A\) is derived \((p, I)\)-complete and Noetherian local, \(A/p A\) is also a Noetherian local ring.
    By Kunz's theorem, \(A/p A\) is a regular local ring, and thus \((A, I)\) is a regular local prism by \Cref{EquivRegularPrism}.

    (\ref{PrismaticKunzCor-FFlat}) \(\Rightarrow\) (\ref{PrismaticKunzCor-ModIFFlat}): The induced map \(\map{\overline{\varphi}_{A,I}}{A/I}{\varphi_*(A/\varphi_A(I)A)}\) is the base change of the faithfully flat map \(\map{\varphi_A}{A}{\varphi_*A}\) via \(A \to A/I\).
    Thus \(\overline{\varphi}_{A,I}\) is faithfully flat.
    By the above (\ref{PrismaticKunzCor-FFlat}) \(\Leftrightarrow\) (\ref{PrismaticKunzCor-RegPrism}), \((A, I)\) is transversal or crystalline.

    (\ref{PrismaticKunzCor-ModIFFlat}) \(\Rightarrow\) (\ref{PrismaticKunzCor-ModIpFFlat}) \(\Rightarrow\) (\ref{PrismaticKunzCor-IFFlat}): This is a direct consequence of \Cref{TransversalFFlatEquiv}.
\end{proof}

\begin{remark}
    Applying this for \Cref{LogRegularPrism}, the regularity of a complete log regular ring \(R \cong W[|\mcalQ|]/(p-f)\) can be characterized by the faithful flatness of the Frobenius lift \(\map{\varphi}{W[|\mcalQ|]}{W[|\mcalQ|]}\) given by \(\varphi(e^q) = (e^q)^p\).
\end{remark}

\begin{remark}
    In \cite[Theorem 6]{lurie2023Full}, Lurie shows that, for a Noetherian ring \(R\) admitting a regular element \(\pi \in R\) such that \(\pi^p\) divides \(p\), the localization \(R_\mfrakm\) for any maximal ideal \(\mfrakm\) of \(R\) containing \(\pi\) is a regular local ring if and only if the \(p\)-th power map \(R/\pi R \xrightarrow{a \mapsto a^p} R/\pi^p R\) is faithfully flat.
    The proof of ``only if part'', which is the easy but necessary part in his paper, can be deduced from the implication \IMPLIES{(\ref{PrismaticKunzCor-RegPrism})}{(\ref{PrismaticKunzCor-ModIFFlat})} in \Cref{PrismaticKunzCor} and a simple calculation in the last paragraph of his proof.
    So a natural question is whether it is possible to prove the ``if part'' using prismatic Kunz's theorem (\Cref{AnimatedCase}).
\end{remark}

Next, by using this, we give another proof of the fact that for any regular local ring \(R\) of residue characteristic \(p\) and any prime ideal \(\mfrakp\) of \(R\) with \(p \in \mfrakp\) the localization \(R_{\mfrakp}\) is also a regular local ring.
This is a generalization of \cite[Corollary 2.2]{kunz1969Characterizations}, which proves the same statement for regular local rings of characteristic \(p\) by using Kunz's theorem.
Note that, if we show the above \Cref{PrismaticKunzCor} under the second proof, we do not use Serre's regularity criterion and \(p\)-adic Kunz's theorem (\Cref{KunzRegular}) to prove this statement.

\begin{proposition} \label{KunzAnotherProof}
    Let \((R, \mfrakm)\) be a regular local ring of residue characteristic \(p\) and \(\mfrakp\) be a prime ideal of \(R\) such that \(p \in \mfrakp\).
    Then the localization \(R_{\mfrakp}\) is also a regular local ring.
\end{proposition}

\begin{proof}
    Let \(\widehat{R}\) be the \(\mfrakm\)-adically completion of \(R\).
    Then \(\widehat{R}\) is a complete regular local ring of residue characteristic \(p\) by \citeSta{07NY}.
    Since \(R \to \widehat{R}\) is faithfully flat, there exists a prime ideal \(\mfrakq\) of \(\widehat{R}\) such that \(\mfrakq \cap R = \mfrakp\).
    Then we have a map \(R_\mfrakp \to \widehat{R}_\mfrakq\) and this is flat by \cite[Theorem 7.1]{matsumura1986Commutative} and the induced map \(\Spec(\widehat{R}_\mfrakq) \to \Spec(R_\mfrakp)\) is surjective by \cite[Lemma 14.9]{gortz2010Algebraic}.
    Then \(R_\mfrakp \to \widehat{R}_\mfrakq\) is faithfully flat and it suffices to show that \(\widehat{R}_\mfrakq\) is a regular local ring.
    Without loss of generality, we can assume that \(R\) is a complete regular local ring.

    By \Cref{EquivRegularRing}, there exists a complete regular local prism \((A, I)\) such that \(R \cong A/I\).
    We can consider \(\mfrakp\) as a prime ideal of \(A\) which contains \(p\) and \(I\).

    Take an element \(a \in A\) such that \(\varphi(a) \in \mfrakp\).
    Then \(\varphi(a) = a^p + p \delta(a)\) and \(p\) are in \(\mfrakp\) and thus \(a\) is in \(\mfrakp\).
    This shows that \(\varphi(A \setminus \mfrakp) \subset A \setminus \mfrakp\).
    By \cite[Lemma 2.15]{bhatt2022Prismsa}, the localization \(A_\mfrakp\) admits a \(\delta\)-structure which makes the canonical map \(A \to A_\mfrakp\) a map of \(\delta\)-rings.
    In particular, the \(\delta\)-pair \((A_\mfrakp, I_\mfrakp)\) forms an orientable prism.

    By using \Cref{PrismaticKunzCor}, it suffices to show that the Frobenius map \(\map{F}{(A/pA)_\mfrakp}{F_*((A/pA)_\mfrakp)}\) is faithfully flat.
    Since the canonical map \(A \to A_\mfrakp\) is flat and the Frobenius lift \(\map{F}{A/pA}{F_*(A/pA)}\) is faithfully flat by regularity of \((A, I)\) and \Cref{PrismaticKunzCor}, the base change map
    \begin{equation} \label{BasechangeFrobLift}
        (A/pA)_\mfrakp \cong A_\mfrakp \otimes^L_A A/pA \xrightarrow{\id_{A_\mfrakp} \otimes^L F} A_\mfrakp \otimes^L_A F_*(A/pA) \cong (F_*(A/pA))_\mfrakp \cong F_*((A/pA)_\mfrakp)
    \end{equation}
    is again faithfully flat and we finish the proof.
\end{proof}

\appendix
\section{Preliminaries about Animated Rings} \label{AnimatedAppendix}

For convenience, we summarize a few basic knowledge of \emph{animated rings} which is much more than we actually need in the paper.
The main references are \cite{lurie2017Higher,lurie2018Spectral,cesnavicius2024Purity,yaylali2022Notes,khan2023Lectures,bhatt2022Absolute}.
Let \(k\) be a ring and \(\Poly_k\) be the full subcategory of the category of rings whose objects are finite variable polynomial rings over \(k\).
We write \(\mcalS\) as the \emph{\(\infty\)-category of spaces} (also called \emph{\(\infty\)-groupoids \(\Grpd_\infty\)} or \emph{anima \(\Ani\)}) and \(\Spectrum\) as the \emph{\(\infty\)-category of spectra}.

\begin{definition}[Animation of categories; {\cite[\S 5.1]{cesnavicius2024Purity}}] \label{AnimationCategory}
    Let \(\mcalC\) be a cocomplete (1-)category generated under colimits by the full subcategory \(\mcalC^{\mathrm{sfp}}\) of strongly of finite presentation objects.
    For example, \(\mcalC\) can be taken as the category of sets \(\Set\), abelian groups \(\Ab\), and (commutative and unital) \(k\)-algebras \(\CAlg_k\).
    Then the \emph{animation} of \(\mcalC\) is the pair \((\Ani(\mcalC), \mcalC^{\mathrm{sfp}} \hookrightarrow \Ani(\mcalC))\) of an \(\infty\)-category \(\Ani(\mcalC)\) which has sifted colimits and a fully faithful functor \(\mcalC^{\mathrm{sfp}} \hookrightarrow \Ani(\mcalC)\) such that the post composition yields an equivalence
    \begin{equation*}
        \Fun_{\mathrm{sifted}}(\Ani(\mcalC), \mcalA) \xrightarrow{\sim} \Fun(\mcalC^{\mathrm{sfp}}, \mcalA)
    \end{equation*}
    for any \(\infty\)-category \(\mcalA\) which has sifted colimits. Here, \(\Fun_{\mathrm{sifted}}(\Ani(\mcalC), \mcalA)\) is the full subcategory of \(\Fun(\Ani(\mcalC), \mcalA)\) spanned by those functors which preserve sifted colimits (or, equivalently, preserve filtered colimits and geometric realizations by \cite[Corollary 5.5.8.17]{lurie2009Higher}).
    The animation is uniquely determined up to (unique) equivalence and we denote it simply by \(\Ani(\mcalC)\).
    If \(\mcalC^{\mathrm{sfp}}\) is small, \(\Ani(\mcalC)\) is equivalent to the \(\infty\)-category \(\Fun_{\Sigma}((\mcalC^{\mathrm{sfp}})^{\opposite}, \mcalS)\) of functors \((\mcalC^{\mathrm{sfp}})^{\opposite} \to \mcalS\) which takes finite coproducts in \(\mcalC^{\mathrm{sfp}}\) to finite products in \(\mcalS\) and the Yoneda embedding \(\mcalC^{\mathrm{sfp}} \hookrightarrow \Fun_{\Sigma}((\mcalC^{\mathrm{sfp}})^{\opposite}, \mcalS)\).
\end{definition}

\begin{definition}[Animated rings; {\cite[\S 5.1.4]{cesnavicius2024Purity} and \cite[Definition 25.1.1.1]{lurie2018Spectral}}] \label{AnimatedRings}
    The \emph{\(\infty\)-category \(\ACAlg_k\) of animated (commutative) \(k\)-algebras} is the animation \(\Ani(\CAlg_k)\) of the category \(\CAlg_k\) of (commutative and unital) \(k\)-algebras.
    As explained above, \(\ACAlg_k\) is equivalent to the full subcategory \(\Fun_{\Sigma}(\Poly_k^{\opposite}, \mcalS) \subseteq \Fun(\Poly_k^{\opposite}, \mcalS)\) spanned by those functors which preserve finite products.
    We refer to the objects as \emph{animated (commutative) \(k\)-algebras}.
    If \(k = \setZ\), we denote \(\ACAlg_\setZ\) by \(\ACAlg\) and call it the \emph{\(\infty\)-category of animated (commutative) rings} whose objects are \emph{animated (commutative) rings}.
    We omit the term commutative.
\end{definition}

\begin{Notation}
    First, we define the following (\(\infty\)-)categories.
    The first two appear in the paragraph above \cite[Proposition 2.34]{yaylali2022Notes} and the third and fourth appear in \cite[Notation 8.1.4 and Definition/Proposition 8.1.5]{khan2023Lectures}.
    \begin{itemize}
        \item \(\Mod(\Spectrum)_k\) is the \(\infty\)-category of pairs \((A, M)\), where \(A\) is in the \(\infty\)-category \(\CAlg_k(\Spectrum)\) of \(E_\infty\)-\(k\)-algebras and \(M\) is an \(A\)-module.
        \item Set the \(\infty\)-category \(\ACAlgMod_k \defeq \Mod(\Spectrum)_k \times_{\CAlg_k(\Spectrum)} \ACAlg_k\) and let \(\ACAlgMod_k^{cn}\) be the full subcategory of \(\ACAlgMod_k\) consisting of objects \((A, M)\) such that \(M\) is connective.
        \item \(\CAlgMod_k\) is the category of pairs \((A, M)\), where \(A\) is a usual \(k\)-algebra and \(M\) is an \(A\)-module.
        \item \(\AMod_A\) is the fiber of \(\Ani(\CAlgMod_k) \xrightarrow{\pi} \Ani(\CAlg_k) = \ACAlg_k\) over \(A \in \ACAlg_k\).
    \end{itemize}
\end{Notation}

\begin{definition}[Animated modules; {\cite[\S 5.1.7]{cesnavicius2024Purity}}] \label{DefAnimatedModule}
    Let \(A\) be an animated ring.
    The \emph{\(\infty\)-category \(\Mod_A\) of \(A\)-modules} is the (symmetric monoidal) \(\infty\)-category of modules over the underlying \(E_\infty\)-ring of \(A\).
    The \emph{\(\infty\)-category \(\Mod_A^{cn}\) of animated \(A\)-modules} is the (symmetric monoidal) full subcategory of \(\Mod_A\) spanned by the connective modules over the underlying \(E_\infty\)-ring of \(A\).
    By using the above notations, these \(\infty\)-categories are defined as follows:
    \begin{align*}
        \Mod_A & \defeq \fib(\ACAlgMod_\setZ \to \{A\} \subseteq \ACAlg_\setZ), \\
        \Mod_A^{cn} & \defeq \fib(\ACAlgMod_\setZ^{cn} \to \{A\} \subseteq \ACAlg_\setZ).
    \end{align*}
    We call the objects in \(\Mod_A^{cn}\) as \emph{animated \(A\)-modules}.
\end{definition}

Animated modules have the following properties and another construction.

\begin{lemma} \label{AnimatedModule}
    Let \(A\) be an animated \(k\)-algebra.
    Let \(\mscrC\) be the full subcategory of \(\ACAlgMod_k^{cn}\) consisting of pairs \((A, M)\) where \(A \in \Poly_k\) and \(M\) is a finite free \(A\)-module.
    \begin{enumerate}
        \item The objects of \(\mscrC\) form compact projective generators of \(\ACAlgMod_k^{cn}\) by \cite[Proposition 25.2.1.2]{lurie2018Spectral}
        That is, \(\mscrC \hookrightarrow \ACAlgMod_k^{cn}\) extends to an equivalence \(\Fun_\Sigma(\mscrC^{op}, \mcalS) \cong \ACAlgMod_k^{cn}\).
        \item In particular, we have an equivalence of \(\infty\)-categories \(\ACAlgMod_k^{cn} \simeq \Ani(\CAlgMod_k)\) (see the paragraph above \cite[Proposition 2.34]{yaylali2022Notes}).
        \item The \(\infty\)-categories \(\AMod_A\) and \(\Mod_A^{cn}\) are equivalence.
        \item The \(\infty\)-category \(\Mod_A\) is a stable \(\infty\)-category.
    \end{enumerate}
\end{lemma}

The notions of \emph{derived completeness} and \emph{derived completion} for modules over a \emph{discrete} ring are introduced in \Cref{DefDerivedComplete}.
However, we need to generalize them to modules over an animated ring for our purpose.
The main reference is \cite[\S 7.1, \S 7.2, \S 7.3]{lurie2018Spectral}.

\begin{definition}[Derived \(I\)-complete modules] \label{DefDerivedIComplete}
    Let \(A\) be an animated ring and let \(I\) be a finitely generated ideal of \(\pi_0(A)\).
    Let \(M\) be an \(A\)-module.
    \begin{itemize}
        \item \(M\) is \emph{\(I\)-nilpotent} if the derived tensor product \(A[1/x] \otimes^L_A M\) vanishes for each \(x \in I\), or equivalently, the action of \(x \in I\) on \(\pi_i(M)\) is locally nilpotent.\footnote{The action of \(x \in I\) on a discrete \(\pi_0(A)\)-module \(N\) is \emph{locally nilpotent} if for each \(m \in N\), there exists \(k \geq 1\) such that \(x^k m = 0\).}
        \item \(M\) is \emph{\(I\)-local} if the mapping space \(\Map_A(N, M)\) is contractible for every \(I\)-nilpotent object \(N \in \Mod_A\).
        \item \(M\) is \emph{derived \(I\)-complete} if the mapping space \(\Map_A(N, M)\) is contractible for every \(I\)-local object \(N \in \Mod_A\).
        \item \(M\) is derived \(I\)-complete if and only if \(\pi_i(M)\) is a derived \(I\)-complete \(\pi_0(A)\)-module for all \(i \in \setZ\) in the sense of \Cref{DefDerivedComplete} (See also \citeSta{091N}).
        \item The \emph{\(\infty\)-category \(\Mod_A^{\wedge_I}\) of derived \(I\)-complete \(A\)-modules} is the full subcategory of \(\Mod_A\) spanned by derived \(I\)-complete \(A\)-modules. We often write \(\Mod_A^{\wedge_I}\) as \(\Mod_A^\wedge\) if \(I\) is clear from the context.
    \end{itemize}
\end{definition}

\begin{definition}[Derived \(I\)-completion; {\cite[Notation 7.3.1.5 and Proposition 7.3.4.4]{lurie2018Spectral}}] \label{DefDerivedCompletion}
    Let \(A\) be an animated ring and let \(I\) be a finitely generated ideal of \(\pi_0(A)\).
    The \emph{derived \(I\)-completion functor \(\widehat{(-)}^I\)} is the left adjoint of the inclusion \(\Mod_A^{\wedge_I} \hookrightarrow \Mod_A\).
    We often write \(\widehat{(-)}^I\) as \(\widehat{(-)}\) if \(I\) is clear from the context.
    This functor has the following properties.
    \begin{itemize}
        \item This functor is right \(t\)-exact.
        \item If \(A\) is discrete, this coincides with the usual derived \(I\)-completion functor defined in \Cref{DefDerivedComplete} because the functor is also the left adjoint of the inclusion \(\Mod_A^{\wedge_I} \hookrightarrow \Mod_A\) by \citeSta{091V}.
    \end{itemize}
\end{definition}

We freely use the following properties of the derived \(I\)-completion functor.

\begin{lemma}[{\cite[Corollary 7.3.3.6]{lurie2018Spectral}}] \label{DerivedCompletionCompatible}
    Let \(A \to A'\) be a map of animated rings and let \(I\) be a finitely generated ideal of \(\pi_0(A)\).
    Set a finitely generated ideal \(I' \defeq I\pi_0(A')\) in \(\pi_0(A')\).
    \begin{enumerate}
        \item An \(A'\)-module \(M\) is derived \(I'\)-complete if and only if \(M\) is derived \(I\)-complete as an \(A\)-module.
        \item The derived \(I'\)-completion functor \(\widehat{(-)}^{I'}\) on \(\Mod_{A'}\) is equivalent to the restriction of the derived \(I\)-completion functor \(\widehat{(-)}^{I}\) on \(\Mod_A\) to \(\Mod_{A'}\), that is, \(\widehat{M}^{I'} \cong \widehat{M}^I\) in \(\Mod_A\) for any \(M \in \Mod_{A'}\).
    \end{enumerate}
    In particular, for any \(A\)-module \(M\), the derived \(p\)-completion \(\widehat{M}^p\) of \(M\) in \(\Mod_A\) is isomorphic to \(R\lim_n (M \otimes^L_\setZ \Kos(\setZ; p^n))\) as \(A\)-module since the canonical \(A\)-module map between them induces an isomorphism in \(\Mod_\setZ = \mcalD(\setZ)\).
    So we do not distinguish them.
\end{lemma}

\begin{lemma} \label{DerivedCompletionConnectedComponent}
    Let \(A\) be an animated ring and let \(I\) be a finitely generated ideal of \(\pi_0(A)\).
    Let \(M\) be an animated \(A\)-module such that the derived \(I\)-completion \(\widehat{\pi_0(M)}\) of \(\pi_0(M)\) is discrete.
    Then we have
    \begin{equation*}
        \pi_0(\widehat{M}) \cong \widehat{\pi_0(M)}.
    \end{equation*}
    In particular, if \(I = (p)\) and \(\pi_0(M)\) has bounded \(p^\infty\)-torsion, this holds. 
\end{lemma}

\begin{proof}
    We show this statement by using the same argument as in \cite[Corollary 7.3.6.6]{lurie2018Spectral}.
    Taking a fiber sequence \(\tau_{\geq 1}M \to M \to \pi_0(M)\) in \(\Mod_A\), we have a fiber sequence \(\widehat{\tau_{\geq 1}M} \to \widehat{M} \to \widehat{\pi_0(M)}\) in \(\Mod_A^\wedge\).
    Since the derived \(I\)-completion functor \(\widehat{(-)}\) is right \(t\)-exact, \(\widehat{\tau_{\geq 1}M}\) is in \((\Mod_A)_{\geq 1}\).
    So the long exact sequence shows that \(\pi_0(\widehat{M}) \xrightarrow{\cong} \pi_0(\widehat{\pi_0(M)})\) is an isomorphism of \(\pi_0(A)\)-modules.
    Since \(\widehat{\pi_0(M)}\) is discrete in our assumption, \(\pi_0(\widehat{M}) \cong \widehat{\pi_0(M)}\) as desired.

    If \(\pi_0(M)\) has bounded \(p^\infty\)-torsion, \(\widehat{\pi_0(M)}\) is discrete and coincides with the \(p\)-adic completion of \(\pi_0(M)\).
    So \(M\) satisfies the assumption of the above argument.
\end{proof}

\begin{definition}[Derived \(\infty\)-categories] \label{RemarkDerivedInfCat}
    Let \(A\) be an animated ring and let \(I\) be a finitely generated ideal of \(\pi_0(A)\).
    The \emph{derived \(\infty\)-category \(\mcalD(A)\) of \(A\)} is nothing but the (symmetric monoidal) \(\infty\)-category \(\Mod_A\) of \(A\)-modules.
    The \emph{\(I\)-completed derived \(\infty\)-category \(\mcalD_{I-comp}(A)\) (or \(\widehat{\mcalD}(A)\))} is nothing but the full subcategory \(\Mod_A^{\wedge_I}\) of \(\mcalD(A)\) spanned by the derived \(I\)-complete \(A\)-modules defined in \Cref{DefDerivedIComplete}.

    In particular, if \(A\) is discrete, \(\Mod_A\) is equivalent to that the derived \(\infty\)-category \(\mcalD(A)\) of \(A\) and \(\Mod_A^{cn}\) is equivalent to \(\mcalD^{\leq 0}(A)\) in cohomological notation (as symmetric monoidal \(\infty\)-categories; see \cite[Proposition 7.1.2.13]{lurie2017Higher}).
\end{definition}

\begin{definition}[Some properties of animated modules; {\cite[\S 7.2.2]{lurie2017Higher}}] \label{PropOfAnimatedModule}
    Let \(A\) be an animated \(k\)-algebra (or more simply, an \(E_\infty\)-ring) and let \(M\) be an animated \(A\)-module.
    \begin{enumerate}
        \item \(M\) is \emph{free} if it is equivalent to a coproduct of copies of \(A\).
        \item \(M\) is \emph{projective} if it is a projective object of the \(\infty\)-category \(\Mod_A^{cn}\).
        \item \(M\) is \emph{(faithfully) flat} if \(\pi_0(M)\) is a (faithfully) flat \(\pi_0(A)\)-module and the natural map \(\pi_0(M) \otimes_{\pi_0(A)} \pi_i(A) \to \pi_i(M)\) is an isomorphism of abelian groups for each \(i \in \setZ\).
        \item For a finitely generated ideal \(I \subseteq \pi_0(A)\), \(M\) is \emph{\(I\)-completely (faithfully) flat} if the base change \(M \otimes^L_A \pi_0(A) \in \Mod_{\pi_0(A)} = \mcalD(\pi_0(A))\) is \(I\)-completely (faithfully) flat over \(\pi_0(A)\) in the sense of \Cref{DefofCompFlat}.
    \end{enumerate}
\end{definition}

\begin{lemma}[{\cite[\S 7.2.2]{lurie2017Higher}}] \label{PropertiesModule}

    Let \(A\) be an animated \(k\)-algebra and let \(M\) and \(N\) be animated \(A\)-modules.
    We use the following facts:
    \begin{enumerate}
        \item If \(M\) is free, then \(M\) is projective.
        \item If \(M\) is projective, then \(M\) is flat.
        \item If \(M\) is projective, any surjective map of animated \(A\)-modules \(N \to M\) has a right inverse (up to homotopy).
    \end{enumerate}
\end{lemma}

\begin{definition} \label{ExteriorPowerCotangentComplex}
    Let \(\mscrC\) be the full subcategory of \(\ACAlgMod_k^{cn}\) spanned by those pairs \((A, M)\) where \(A \in \Poly_k\) and \(M \cong A^n\) for some \(n \geq 0\).
    For any \(A \in \Poly_k\), the module of K\"ahler differential \(\Omega^1_{A/k}\) of \(A\) over \(k\) defines a functor
    \begin{align*}
        \Omega^i_{-/k} \colon \Poly_k & \rightarrow \mscrC \subseteq \ACAlgMod_k^{cn} \\
        A & \mapsto (A, \Omega^i_{A/k} \defeq \bigwedge^i_A \Omega^1_{A/k})
    \end{align*}
    for each \(i \geq 0\) since \(\Omega^i_{A/k}\) is a free \(A\)-module.
    By \cite[Construction B.1]{bhatt2022Absolute}, \(\Omega^i_{-/k}\) admits a unique extension \(\map{L\Omega^i_{-/k}}{\ACAlg_k}{\ACAlgMod_k^{cn}}\).
    We refer to \(L\Omega^1_{A/k}\) as the \emph{cotangent complex of \(A\) over \(k\)} for any animated \(k\)-algebra \(A\) and denote it by \(L_{A/k}\).
    
    By \cite[Construction 25.2.2.2]{lurie2018Spectral},
    the endofunctor \(\bigwedge^i \colon (A, M) \mapsto (A, \bigwedge^i_A M)\) on \(\mscrC\) has a unique extension \(\map{L\bigwedge^i}{\ACAlgMod_k^{cn}}{\ACAlgMod_k^{cn}}\).
    We refer to \(\bigwedge^i_A(M) \defeq L\bigwedge^i(A, M)\) for \((A, M) \in \ACAlgMod_k^{cn}\) as the \emph{derived \(i\)-th exterior power of \(M\) over \(A\)} for each \(i \geq 0\).
    Even if \(A\) is a discrete ring and \(M\) is a discrete \(A\)-module, the derived \(i\)-th exterior power \(\bigwedge^i_A(M)\) is not necessarily discrete and thus is not equivalent to the usual exterior power of \(M\) over \(A\).
    However, if we take the connected component \(\pi_0(\bigwedge^i_A(M))\), then these are isomorphic (see \cite[Warning 25.2.3.5]{lurie2018Spectral}).

    Then we have two functors \(L\Omega^i_{-/k}\) and \((L\bigwedge^i) \circ (L\Omega^1_{-/k})\) from \(\ACAlg_k\) to \(\ACAlgMod_k^{cn}\) which preserve sifted colimits and compatible on \(\Poly_k (\subseteq \ACAlg_k)\).
    Then \cite[Proposition 25.1.1.5 (1)]{lurie2018Spectral} shows that these two functors are naturally equivalent, that is, \(L\Omega^i_{A/k} \cong \bigwedge^i_A(L\Omega^1_{A/k})\) as animated \(A\)-modules for any animated \(k\)-algebra \(A\).
\end{definition}

\begin{remark}[{cf. \cite[Remark 2.51]{yaylali2022Notes}}] \label{EquivCotangentComplex}
    In \cite[Construction 25.3.1.6 and Notation 25.3.2.1]{lurie2018Spectral}, we have an \emph{algebraic cotangent complex \(L^{\alg}_{S/R}\) of \(S\) over \(R\)} for any map of animated rings \(R \to S\) by using some universal property.
    If \(R\) is discrete, the cotangent complex \(L_{S/R}\) defined in \Cref{ExteriorPowerCotangentComplex} is isomorphic to the algebraic cotangent complex \(L^{\alg}_{S/R}\) as an animated \(S\)-module since \(L_{S/R}\) satisfies the same universal property of \(L^{\alg}_{S/R}\) by a direct calculation \cite[Proposition 2.50]{yaylali2022Notes}.
\end{remark}

\end{document}